\input ./style/arxiv-general.cfg
\documentclass[aop,MSNbibl,nameyear,seceqn,dvips]{arximspdf}
\makeatletter
   \@ifpackageloaded{graphicx}{}{\usepackage{graphicx}}
\makeatother

\doi{10.1214/15-AOP1033}
\volume{44}
\issue{4}
\pubyear{2016}
\firstpage{2726}
\lastpage{2769}
\docsubty{FLA}

\makeatletter
\newcommand{\lleft}{\left}
\newcommand{\rrvert}{\vert}
\newcommand{\rright}{\right}
\newcommand{\rrVert}{\Vert}
\newcommand{\llvert}{\vert}
\newcommand{\llVert}{\Vert}
\renewcommand{\mid}{|}
\newcommand{\ph}{\varphi}
\newcommand{\del}{\partial}
\newcommand{\Z}{\mathbb Z}
\newcommand{\R}{\mathbb R}
\newcommand{\C}{\mathbb C}
\newcommand{\F}{\mathbb F}
\newcommand{\HH}{\mathbb H}
\newcommand{\pp}{\mathbf p}
\newcommand{\ww}{\mathbf w}
\newcommand{\Cc}{\mathcal C}
\newcommand{\Hc}{\mathcal H}
\newcommand{\Pc}{\mathcal P}
\newcommand{\diag}{\operatorname{diag}}
\newcommand{\Span}{\operatorname{span}}
\renewcommand{\P}{\operatorname{\mathbf{P}}}
\newcommand{\1}{\mathbf{1}}
\newcommand{\E}{\mathbf E}
\renewcommand{\Re}{\operatorname{Re}}

\newcommand{\wt}{\widetilde}
\newcommand{\ol}{\overline}
\newcommand{\Ai}{\operatorname{Ai}}

\newtheorem{theorem}{Theorem}[section]
\newtheorem{lemma}[theorem]{Lemma}
\newtheorem{proposition}[theorem]{Proposition}
\newtheorem{fact}[theorem]{Fact}
\newproclaim{definition}[theorem]{Definition}
\newproclaim{remark}[theorem]{Remark}
\newproclaim{ass}{Assumption}
\makeatother

\begin{document}
\begin{frontmatter}

\title{Limits of spiked random matrices II}
\runtitle{Spiked random matrices}

\begin{aug}
\author[A]{\fnms{Alex}~\snm{Bloemendal}\corref{}\ead[label=e1]{alexb@math.harvard.edu}\thanksref{T1}}
\and
\author[B]{\fnms{B\'alint}~\snm{Vir\'ag}\ead[label=e2]{balint@math.toronto.edu}\thanksref{T2}}
\runauthor{A. Bloemendal and B. Vir\'ag}
\affiliation{Harvard University and University of Toronto}
\address[A]{Department of Mathematics\\
Harvard University\\
Cambridge, Massachusetts 02138\\
USA\\
\printead{e1}}
\address[B]{Departments of Mathematics and Statistics\\
University of Toronto\\
Toronto, Ontario M5S 2E4\\
Canada\\
\printead{e2}}
\end{aug}
\thankstext{T1}{Supported by an NSERC postgraduate scholarship held at the University of Toronto.}
\thankstext{T2}{Supported  by the Canada Research Chair program and the NSERC DAS program.}

%
\received{\smonth{6} \syear{2012}}
%
\revised{\smonth{5} \syear{2015}}

%
\begin{abstract}
The top eigenvalues of rank $r$ spiked real Wishart matrices and
additively perturbed Gaussian orthogonal ensembles are known to exhibit
a phase transition in the large size limit. We show that they have
limiting distributions for near-critical perturbations, fully resolving
the conjecture of Baik, Ben Arous and P\'ech\'e
[\textit{Duke Math. J.} (2006) \textbf{133} 205--235]. The starting
point is a new $(2r+1)$-diagonal form that is algebraically natural to
the problem; for both models it converges to a certain random Schr\"
odinger operator on the half-line with $r\times r$ matrix-valued
potential. The perturbation determines the boundary condition and the
low-lying eigenvalues describe the limit, jointly as the perturbation
varies in a fixed subspace. We treat the real, complex and quaternion
($\beta= 1,2,4$) cases simultaneously. We further characterize the
limit laws in terms of a diffusion related to Dyson's Brownian motion,
or alternatively a linear parabolic PDE; here $\beta$ appears simply
as a parameter. At $\beta= 2$, the PDE appears to reconcile with known
Painlev\'e formulas for these $r$-parameter deformations of the GUE
Tracy--Widom law.
\end{abstract}

%
\end{frontmatter}

\section{Introduction}\label{sec1}

Johnstone (\citeyear{J1}) proposed the spiked population model for simple trends
in high dimensional data. One takes a data matrix $X$ whose columns are
i.i.d. vectors with (population) covariance a fixed rank perturbation
of the identity, and studies the behaviour of the largest eigenvalues
of the sample covariance matrix $XX^*$ when both the dimension and the
size of the sample are large. \citet{BBP} (hereafter \textit{BBP}) discovered a very interesting phase transition phenomenon in the
complex Gaussian setting. Small spikes do not affect the asymptotic
behaviour of the top eigenvalues, which display the usual Tracy--Widom
fluctuations around the upper edge of the Marchenko--Pastur law; large
spikes, however, lead to outliers with Gaussian fluctuations. New
structure emerges near the transition point with near-critical spikes
deforming the soft edge limit. Understanding this transition regime in
the real case remained open for some time. There is a parallel
development for fixed rank additive perturbations of Wigner matrices.


In \citet{BV1} (hereafter \textbf{Part}~I), we considered rank
one spiked real/complex/quaternion Wishart matrices and additive rank
one perturbations of the Gaussian orthogonal, unitary and symplectic
ensembles. Our approach is based on the continuum operator limit at the
general beta soft edge developed in \citet{RRV} (hereafter
\textit{RRV}).
We introduced general $\beta$ analogues of the rank one spiked models,
modifying the tridiagonal ensembles of \citet{DE} and extended
the RRV technology to describe the soft-edge scaling limit in terms of
the stochastic Airy operator
\[
-\frac{d^2}{dx^2}+\frac{2}{\sqrt\beta}b_x'+x
\]
on $L^2(\R_+)$ with a boundary condition depending on the spike. The
boundary condition changes from Dirichlet $f(0)=0$ to Neumann/Robin
$f'(0)=wf(0)$ at the onset of the BBP phase transition, with $w\in\R$
representing a scaling parameter for perturbations in a ``critical
window''. The resulting largest eigenvalue laws form a one-parameter
family of deformations of Tracy--Widom($\beta$), naturally
generalizing the characterization of
{RRV} in terms of the ground state of this random Schr\"
odinger operator.

We went on to characterize the limit laws in terms of the diffusion
from {RRV} and in terms of an associated second-order
linear parabolic PDE. We further showed that at $\beta= 2,4$ the PDE
is related to known Painlev\'e II representations originating in
\citet{BR1} and gave new proofs of these, finally recovering
those of the undeformed Tracy--Widom laws.

Even the existence of limiting distributions in the critical regime was
in general new for $\beta\neq2$, though see the prior work of
\citet{W} on the rank one \mbox{$\beta=4$} case at $w=0$, as well as
the subsequent work of \citet{Mo} offering a more standard
treatment of the rank one $\beta=1$ case. \citet{F3} comments on
all three works and gives an alternative interpretation and
construction of our general $\beta$ rank one spiked model.

Here, we deal with $r$ ``spikes'', or general bounded-rank
perturbations of Gaussian and Wishart matrices. To do so, we introduce
a new ``canonical form for perturbations in a fixed subspace'', a
$(2r+1)$-diagonal band form that has a purely algebraic interpretation.
It generalizes the Dumitriu--Edelman forms and is able to handle rank
$r$ perturbations. We then develop a generalization of the methods of
{RRV} and {Part}~I  to a matrix-valued setting:
block tridiagonal matrices converge to a half-line Schr\"odinger
operator with matrix-valued potential, the spikes once again appearing
in the boundary condition. We treat the real, complex and quaternion
($\beta= 1,2,4$) cases simultaneously. Once again, even the existence
of a near-critical soft-edge limit is new off $\beta= 2$. Unlike in
Part~I, however, we do not define a general $\beta$ version
of either matrix model, nor of the limiting operator; in Section~\ref
{s.canonical}, we will see that the higher rank versions of these
objects do not readily admit a $\beta$-generalization.

Dyson's Brownian motion makes a surprise appearance, providing nice SDE
and PDE characterizations of the limit laws---new $r$ parameter
deformations of Tracy--Widom($\beta$)---in which $\beta$ reappears as
a simple parameter. The derivation makes use of the matrix-valued
version of classical Sturm oscillation theory and the Riccati
transformation. In a short final section, we report on preliminary
evidence that at $\beta=2$ the PDE can be connected with a Painlev\'e
II representation of \citet{B} for these distributions (which
appeared originally in {BBP} in the form of Fredholm determinants).

We highlight two more features of our approach beyond the novelty of
bypassing formulas for joint eigenvalue densities and handling $\beta
=1,2,4$ together. First, we treat the perturbation as a \emph
{parameter}. By this, we mean that all perturbations in a fixed
subspace are considered jointly (on the same probability space); this
picture is carried through to the limit, which is therefore a family of
point processes parameterized by an $r\times r$ matrix. Second, we
allow more general scalings than those considered in {BBP}.
Most importantly, in the Wishart case we do not require the two
dimensional parameters $n,p$ to have a positive limiting ratio but
rather allow them to tend to infinity together arbitrarily.

To state our results, we introduce some objects and notation that will
be used throughout the paper.

Let $\F= \R$, $\C$, or $\HH$ and $\beta= 1,2$ or $4$,
respectively. A \textit{standard $\F$ Gaussian} $Z\sim\F N(0,1)$ is
an $\F$-valued random variable described in terms of independent real
Gaussians $g_1,\ldots,g_\beta\sim N(0,1)$ as $g_1$ for $\F=\R$,
$(g_1+g_2i)/\sqrt{2}$ for $\F=\C$, and $(g_1+g_2i+g_3j+g_4k)/2$ for
$\F=\HH$. Note that in each case $\E\llvert Z\rrvert^2=1$
and $uZ\sim\F N(0,1)$ for $u\in\F$ with $\llvert u\rrvert
^2 = u^*u =1$.

The space of column vectors $\F^n$ is endowed with the standard inner
product $u^\dag v$ and associated norm $\llvert u\rrvert^2 =
u^\dag u$ (we reserve double bars for function spaces). Write $\F
N_n(0,I)$ for a vector of independent standard $\F$ Gaussians. With
$\Sigma\in M_n(\F)$ positive definite, we write $Z\sim\F
N_n(0,\Sigma)$ for $Z = \Sigma^{1/2}Z_0$ with $Z_0\sim\F N_n(0,I)$.

Define the \textit{unitary group} $U_n(\F) = \{U\in\F^{n\times
n}:U^\dag U = I\}$, better known as the orthogonal, unitary or symplectic
group for $\F= \R, \C, \HH$, respectively. It acts on $\F^n$ by
left multiplication, on which the distribution $\F N_n(0,I)$ is
invariant. Write $M_n(\F) = \{A\in\F^{n\times n}:A^\dag= A\}$ for the
\textit{self-adjoint matrices}, also known as real symmetric, complex
Hermitian or quaternion self-dual. $U_n(\F)$ acts on $M_n(\F)$ by conjugation.

The \textit{Gaussian orthogonal/unitary/symplectic ensemble} (GO/U/SE)
is the probability measure on $M_n(\F)$ described by $A = (X+X^\dag
)/\sqrt{2}$ where $X$ is an $n\times n$ matrix of independent $\F
N(0,1)$ entries. The distribution is invariant under the unitary
action. Furthermore, the algebraically independent entries $A_{ij}$,
$i\ge j$ are statistically independent. (Together, this invariance and
independence characterizes the distribution up to a scale factorr.) For
an entry-wise description, the diagonal entries are distributed as
$N(0,2/\beta)$ while the off-diagonal entries are $\F N(0,1)$.

Fixing a positive integer $r$, we study \textit{rank $r$ additive
perturbations} $A=A_0+P$ of a GO/U/SE matrix $A_0$, where $P = \tilde
P\oplus0_{n-r}$ with $\tilde P\in M_r(\F)$ nonrandom. We will be
interested in the eigenvalues $\lambda_1\ge\cdots\ge\lambda_n$ of
$A$. Of course for a single $P$ their distribution depends only on the
eigenvalues of $P$, but we consider them jointly over all $\tilde P$.

We also consider \textit{real/complex/quaternion Wishart matrices}.
These are random nonnegative matrices in $M_p(\F)$ given by $XX^\dag$
where the \textit{data matrix} $X$ is $p\times n$ with $n$ independent
$\F N_p(0,\Sigma)$ columns. We speak of a $p$-\textit{variate}
Wishart with $n$ \textit{degrees of freedom} and $p\times p$ \textit
{covariance} $\Sigma>0$. Since we are interested in the \emph
{nonzero} eigenvalues $\lambda_1\ge\cdots\ge\lambda_{n\wedge p}$,
we can equally well consider $X^\dag X$. The distribution of $X^\dag X$ may
also be described as $X_0^\dag\Sigma X_0$ where $X_0$ is a $p\times n$
matrix of independent $\F N(0,1)$ entries. The case $\Sigma= I$ is
referred to as the \textit{null case}. We study the \textit{rank $r$
spiked case} where $\Sigma= \tilde\Sigma\oplus I_{p-r}$ with $\tilde
\Sigma\in M_r(\F)$ nonrandom. Once again the eigenvalue distribution
depends only on the eigenvalues of $\Sigma$, but we consider the
spectrum jointly as $\tilde\Sigma$ varies.

Our starting point is a new banded or multi-diagonal form introduced in
Section~\ref{s.canonical}, ideally suited to the types of
perturbations we consider. It is defined for almost every matrix $A\in
M_n(\F)$; given vectors $v_1,\ldots,v_r\in\F^n$, the new basis may
be obtained by applying the Gram--Schmidt process to the first $n$
vectors of the sequence
\[
v_1,\ldots,v_r,Av_1,\ldots,Av_r,A^2v_1,
\ldots,A^2v_r,\ldots.
\]
The result is a $(2r+1)$-diagonal matrix with positive outer diagonals.
For Gaussian and null Wishart ensembles, the change of basis interacts
well with the Gaussian structure; this observation goes back to
\citet{Trotter} in the $r=1$ case. In the GO/U/SE case, we take
$v_1,\ldots,v_r$ to be the initial coordinate basis vectors, while in
the Wishart case we use the initial rows of the data matrix $X$. As in
Part~I, the key observation is then that the perturbations
commute with the change of basis.

For the (unperturbed) Gaussian ensembles, the band form looks like
\[
\lleft[\matrix{ \wt g & g^* &\cdots& g^* & \chi
\cr
g & \wt g & g^* &\cdots&
g^* &\chi
\cr
\vdots& g &\wt g & g^*&\cdots&g^*&\chi
\cr
g & \vdots&g &\ddots&
\ddots& &\ddots&\ddots
\cr
\chi& g &\vdots&\ddots
\cr
& \chi&g
\cr
& & \chi&\ddots
\cr
& & & \ddots} \rright],
\]
where the entries are independent random variables up to the $\dag
$-symmetry with $\wt g\sim N(0,2/\beta)$, $g\sim\F N(0,1)$, and
$\chi\sim\frac{1}{\sqrt{\beta}}\operatorname{Chi}((n-r-k)\beta)$,
$k=0,1,2,\ldots$ going\vspace*{1pt} down the matrix. [Recall that if $Z\sim\R
N_m(0,I)$ then $\llvert Z\rrvert\sim\operatorname{Chi}(m)$.]
For the null Wishart ensemble, the form is best described as follows.
One first obtains a lower $(r+1)$-diagonal form for the data matix $X$
whose nonzero \emph{singular values} are the same as those of $X$. It
looks like
\[
\lleft[\matrix{ \wt\chi
\cr
g & \wt\chi
\cr
\vdots& g &\wt\chi
\cr
g & \vdots&g &
\ddots
\cr
\chi& g &\vdots&\ddots
\cr
& \chi&g
\cr
& & \chi&\ddots
\cr
& & & \ddots
} \rright],
\]
where the entries are independent random variables with $g\sim\F
N(0,1)$, $\wt\chi\sim\frac{1}{\sqrt{\beta}}\operatorname{Chi}((n-k)\beta)$ and $\chi\sim\frac{1}{\sqrt{\beta}}
\operatorname{Chi}((n-r-k)\beta)$, $k=0,1,2,\ldots$ going\vspace*{1pt} down the matrix. One
then forms its multiplicative symmetrization, a $(2r+1)$-diagonal
matrix with the same nonzero eigenvalues as $X$. In both cases, the
perturbations appear in the upper-left $r\times r$ block. Section~\ref
{s.canonical} provides derivations. The obstacle to $\beta
$-generalization at this level is the presence of $\F$ Gaussians in
the intermediate diagonals.

Proceeding with an analogue of the {RRV} convergence result
hinges on reinterpreting these forms as \emph{block tridiagonal} with
$r\times r$ blocks. In Section~\ref{s.limits}, we develop an $M_r(\F
)$-valued analogue of the RRV technology, providing general conditions
under which the principal eigenvalues and corresponding eigenvectors of
such a random block tridiagonal matrix converge to a those of a
continuum half-line random Schr\"odinger operator with matrix-valued
potential. As in Part~I, we allow for a general boundary
condition at the origin.

In Section~\ref{s.clt}, we apply this result to the band forms just
described, proving a process central limit theorem for the potential
and verifying the required tightness assumptions. The limiting operator
turns out to be a multidimensional version of the stochastic Airy
operator, which we now describe.

First, a \textit{standard $\F$ Brownian motion} $\{b_t\}_{t\ge0}$ is
a continuous $\F$-valued random process with $b_0 = 0$ and independent
increments $b_t-b_s\sim\F N(0,t-s)$. (It can be described in terms of
$\beta=1$, 2 or 4 independent standard real Brownian motions.) A
\textit{standard matrix Brownian motion} $\{B_t\}_{t\ge0}$ has
continuous $M_n(\F)$-valued paths with $B_0 = 0$ and independent
increments $B_t - B_s$ distributed as $\sqrt{t-s}$ times a GO/U/SE.
The diagonal processes are thus $\sqrt{2/\beta}$ times standard real
Brownian motions while the off-diagonal processes are standard $\F$
Brownian motions, mutually independent up to symmetry.


Finally, we define the \textit{multivariate stochastic Airy operator}.
Operating on the vector-valued function space $L^2(\R_+,\F^r)$ with
inner product $ \langle f,g \rangle= \int_0^{\infty}
f^\dag g$ and associated norm $\llVert f\rrVert^2 = \int_0^{\infty
}\llvert f\rrvert^2$, it is the random Schr\"
odinger operator
%
%
\begin{equation}
\label{msA} \Hc_\beta= -\frac{d^2}{dx^2} + {\sqrt2}B'_x
+ rx,
\end{equation}
where $B'_x$ is ``standard matrix white noise'', the derivative of a
standard matrix Brownian motion, and $rx$ is scalar. (Here, again
$\beta$ is restricted to the classical values, as the noise term lacks
a straightforward $\beta$-generalization.) The potential is thus the
derivative of a continuous matrix-valued function; rigorous definitions
will appear in Section~\ref{s.limits} in a more general setting.

For now it is enough to know that, together with a general self-adjoint
boundary condition
%
%
\begin{equation}
\label{msABC} f'(0) = Wf(0),
\end{equation}
the multivariate stochastic Airy operator is bounded below with purely
discrete specturm given by a variational principle. Here, $W\in M_r(\F
)$; actually, writing the spectral decomposition $W =\sum_{i=1}^{r}
w_i u_i u_i^\dag$, we formally allow $w_i\in(-\infty,\infty]$.
Writing $f_i = u_i^\dag f$, (\ref{msABC}) is then to be interpreted as
\[
\begin{aligned} f'_i(0) &=& w_i
f_i(0) \qquad\mbox{for }w_i\in\R,
\\
f_i(0) &=& 0 \qquad\mbox{for }w_i=+\infty.
\end{aligned}
\]
We write $W\in M_r^*(\F)$ for this extended set and $\Hc_{\beta,W}$
for (\ref{msA}) together with (\ref{msABC}).

For concreteness, we record that the eigenvalues $\Lambda_0\le\Lambda
_1\le\dots$ and corresponding eigenfunctions $f_0,f_1,\dots$ of $\Hc
_{\beta,W}$ are given, respectively, by the minimum and any minimizer in
the recursive variational problem
\[
\mathop{\inf_{f\in L^2(\R_+)}}_{\llVert f\rrVert=1, f\perp f_0,\ldots
,f_{k-1}} \int_0^\infty
\bigl(\bigl\llvert f'\bigr\rrvert^2 + rx\llvert f
\rrvert^2 \bigr)\,dx +f(0)^\dag Wf(0) + \frac{2}{\sqrt\beta}\int
_0^\infty f^\dag \,dB_x f.
\]
Here, candidates $f$ are only considered if the first integral and
boundary term are finite; the stochastic integral can then be defined
pathwise via integration by parts. The eigenvalues and eigenfunctions
are thus jointly defined random processes indexed over $W$.

%
\begin{remark}\label{r.mon}We note one important property of the
eigenvalue processes, namely the \emph{pathwise monotonicity} of
$\Lambda_k$ in $W$ with respect to the usual matrix partial order.
This is immediate from the variational characterization and the fact
that the objective functional is monotone in $W$. (For the higher
eigenvalues, it is most apparent from the standard min--max formulation
of the variational problem.)
\end{remark}

We can now state the main convergence results. As outlined,
Sections~\ref{s.canonical}--\ref{s.clt} furnish the proofs. One last
shorthand: when we write that a sequence $W_n\in M_r(\F)$ tends to
$W\in M_r^*(\F)$, we mean the following. Writing $W = \sum_{i=1}^r
w_{i}u_i u_i^\dag$ with $w_i\in(-\infty,\infty]$, one has $W_n =
\sum_{i=1}^r w_{n,i}u_i u_i^\dag$ with $w_{n,i}\in\R$ satisfying
$w_{n,i}\to w_i$ for each $i$. In other words, the matrices are
simultaneously diagonal and the eigenvalues tend to the corresponding limits.

%
\begin{theorem}\label{t.G} Let $A=A_0+\sqrt{n}P_n$ where $A_0$ is an
$n\times n$ GO/U/SE matrix and $P_n = \tilde P_n\oplus0_{n-r}$ with
$\tilde P_n\in M_r(\F)$, and let $\lambda_1\ge\cdots\ge\lambda_n$
be its eigenvalues. If
\[
n^{1/3} (1-\tilde P_n) \to W\in M_r^*(\F)\qquad
\mbox{as }n\to\infty
\]
then, jointly for $k=1,2,\dots$ in the sense of finite-dimensional
distributions,
\[
n^{1/6} (\lambda_k-2\sqrt{n} ) \Rightarrow-\Lambda
_{k-1}\qquad\mbox{as }n\to\infty,
\]
where $\Lambda_0\le\Lambda_1\le\dots$ are the eigenvalues of $\Hc
_{\beta,W}$. Convergence holds jointly over $\{P_n\},W$ satisfying the
condition.
\end{theorem}

%
\begin{theorem}\label{t.W} Consider a $p$-variate
real/complex/quaternion Wishart matrix with $n$ degrees of freedom and
spiked covariance $\Sigma_{n,p}=\tilde\Sigma_{n,p}\oplus I_{p-r}>0$
with $\tilde\Sigma_{n,p}\in M_r(\F)$, and let $\lambda_1\ge\cdots
\ge\lambda_{n\wedge p}$ be its nonzero eigenvalues. Writing $m_{n,p}
= (n^{-1/2}+p^{-1/2})^{-2/3}$, if
\[
m_{n,p} \bigl(1-\sqrt{n/p} (\tilde\Sigma_{n,p}-1 ) \bigr) \to
W\in M_r^*(\F)\qquad\mbox{as }n\to\infty
\]
then, jointly for $k=1,2,\dots$ in the sense of finite-dimensional
distributions,
\[
\frac{m_{n,p}^2}{\sqrt{np}} \bigl(\lambda_k-(\sqrt{n}+\sqrt{p})^2
\bigr) \Rightarrow-\Lambda_{k-1}\qquad\mbox{as }n\to\infty,
\]
where $\Lambda_0\le\Lambda_1\le\dots$ are the eigenvalues of $\Hc
_{\beta,W}$. Convergence holds jointly over $\{\Sigma_{n,p}\},W$
satisfying the condition.
\end{theorem}

%
\begin{remark}
In the band basis described above, we also have joint convergence of
the corresponding eigenvectors to the eigenfunctions of $\Hc_{\beta
,W}$. In detail, the eigenvectors should be embedded in $L^2(\R_+)$ as
step functions with step width $n^{-1/3}$ in the Gaussian case and
$m_{n,p}^{-1}$ in the Wishart case, and convergence is in law with
respect to the $L^2$ norm topology. To be precise, one should use
either subsequences or spectral projections; one could also formulate
the joint eigenvalue-eigenvector convergence in terms of the norm
resolvent topology. See Theorem~\ref{t.conv} and the remark that follows.
\end{remark}

We now give the two promised alternative characterizations of the
limiting eigenvalue laws. Fix $\beta=1,2,4$ and $W\in M_r^*(\F)$ with
eigenvalues $-\infty<w_1\le\cdots\le w_r\le\infty$. Writing $\P$
for the probability measure associated with $\Hc_{\beta,W}$ and its
spectrum $\{\Lambda_0\le\Lambda_1\le\ldots\}$, let
\[
F_{\beta}^{k}(x;w_1,\ldots,w_r) = \P(-
\Lambda_k\le x)
\]
for $k=0,1,\ldots.$ Write\vspace*{1pt} simply $F_\beta= F_\beta^0$ for the ground
state distribution (limiting largest eigenvalue law). Once again, the
generalization from Part~I is not straightforward. The
proofs are contained in Section~\ref{s.alternative}.

%
\begin{theorem}\label{t.SDE} Let $\P_{x_0,(w_1,\ldots,w_r)}$ be the
measure on paths $(p_1,\ldots,p_r):[x_0,\infty)\to(-\infty,\infty
]^r$ determined by the coupled diffusions
%
%
\begin{equation}
\label{SDE} dp_i = \frac{2}{\sqrt{\beta}}\,db_i +
\biggl(rx-p_i^2+\sum_{j\neq
i}
\frac{2}{p_i-p_j} \biggr)\,dx
\end{equation}
with initial conditions $p_i(x_0)=w_i$ and entering into $\{p_1<\cdots
<p_r\}$, where $b_1,\ldots,b_r$ are independent standard Brownian
motions; particles $p_i$ may explode to $-\infty$ in finite time
whereupon they are restarted at $+\infty$. Then
%
%
\begin{equation}
\label{explosions} F_\beta(x;w_1,\ldots,w_r) =
\P_{x/r,(w_1,\ldots,w_r)}\qquad(\mbox{no explosions}).
\end{equation}
More generally,
%
%
\begin{equation}
\label{explosionsh} F_\beta^k(x;w_1,
\ldots,w_r) = \P_{x/r,(w_1,\ldots,w_r)}\qquad(\mbox{at most $k$ explosions}).
\end{equation}
\end{theorem}

We describe the diffusion more carefully in Section~\ref
{s.alternative}, asserting that it determines a law on paths valued in
an appropriate space. Probabilistic arguments lead to the following
reformulation in terms of its generator.

%
\begin{theorem}\label{t.PDE} $F_\beta(x;w_1,\ldots,w_r)$ is the
unique bounded function $F:\R\times\R^r\to\R$ symmetric with
respect to permutation of $w_1,\ldots,w_r$ that satisfies the~PDE
%
%
\begin{equation}
\label{PDE} r\frac{\del F}{\del x}+\sum_{i=1}^r
\biggl(\frac{2}{\beta}\frac
{\del^2F}{\del w_i^2}+\bigl(x-w_i^2
\bigr)\frac{\del F}{\del w_i} \biggr)+\sum_{i<j}
\frac{2}{w_i-w_j} \biggl(\frac{\del F}{\del w_i}-\frac{\del
F}{\del w_j} \biggr) = 0\hspace*{-30pt}
\end{equation}
%
and the boundary conditions
%
%
\begin{eqnarray}
\label{PDEBC1} F&\to&1\qquad\mbox{as }x\to\infty\mbox{ with }w_1,
\ldots,w_r\mbox{ bounded below;}
\\
\label{PDEBC2} F&\to&0\qquad\mbox{as any }w_i\to-\infty\mbox{ with }x
\mbox{ bounded above.}
\end{eqnarray}
Furthermore, $F_\beta$ is ``continuous to the boundary'' as one or
several $w_i\to+\infty$. For subsequent eigenvalue laws $F_\beta
^k(x;w_1,\ldots,w_r)$, (\ref{PDEBC2}) is replaced with the recursive
boundary condition
%
%
\begin{eqnarray}
\label{PDEBCm}
F^k(x;w_1,\ldots,w_r)\to
F^{k-1}\bigl(x^*;w_1^*,\ldots,w_{r-1}^*,+\infty
\bigr)
\nonumber\\[-8pt]\\[-8pt]
\eqntext{\mbox{as }x\to x^*\in\R, w_i\to w_i^*\in\R\mbox{ for
}i=1,\ldots,r-1,\mbox{ and }w_r\to-\infty.}
\end{eqnarray}
\end{theorem}

At $\beta=2$, these distributions were obtained in {BBP}
in the form of Fredholm determinants of finite-rank perturbations of
the Airy kernel. \citet{B} derived Painlev\'e II formulas, and by
a symbolic computation with a computer algebra system we were able to
verify that the latter satisfy the PDE (\ref{PDE}) for $r=2,3,4, 5$;
details are described in Section~\ref{s.painleve}. A pencil-and-paper
proof for all $r$ was found since the initial posting [\citet{BBaik}].


We make two final remarks. From the finite $n$ matrix models it is
clear that the ``rank $r$ deformed'' limiting distributions $F_{\beta
,r}(x;w_1,\ldots,w_r)$ reduce to those for a lower rank $r_0<r$ in the
following way:
\[
F_{\beta,r}(x;w_1,\ldots,w_{r_0},+\infty,\ldots,+
\infty) = F_{\beta,r_0}(x;w_1,\ldots,w_{r_0}).
\]
Unfortunately, this reduction relation is not readily apparent from any
of our characterizations (operator, SDE or PDE).

Lastly, the SDE and PDE characterizations seem to make sense for all
$\beta>0$ (although one has to be careful for $\beta<1$). It would be
interesting to find natural ``general $\beta$ multi-spiked models'' at
finite $n$, interpolating between those studied here at $\beta= 1,2,4$
and generalizing those introduced in Part~I for $r=1$. At
$\beta=2$, perhaps one could discover a relationship with formulas of
\citet{BW}.

\section{A canonical form for perturbation in a fixed subspace}\label{s.canonical}

In Part~I, we observed that the tridiagonal models of
Gaussian and Wishart matrices were amenable to rank one perturbation.
In this section, we introduce a banded (also block tridiagonal)
generalization amenable to higher-rank perturbation. We first describe
it as a natural object of pure linear algebra; we then show how it
interacts with the structure of Gaussian and Wishart random matrices to
produce the band forms displayed in the \hyperref[sec1]{Introduction}.

The basic facts of ``linear algebra over $\F$'', where $\F$ may be
$\R$, $\C$ or the skew field of quaternions $\HH$, are summarized in
Appendix E of \citet{AGZ}. Everything we need (inner product
geometry, self-adjointness, eigenvalues, and the spectral theorem)
simply works over $\HH$ as expected, keeping in mind only that
nonreal scalars may not commute. 

\subsection{The band Jacobi form as an algebraic object}

We present a natural ``canonical form'' for studying perturbations in a
fixed subspace of dimension $r$. It is a $(2r+1)$-diagonal band matrix
generalizing the symmetric tridiagonal Jacobi form, which is the $r=1$
case. The outermost diagonals continue to be positive; however,
intermediate diagonals between the main and outermost ones are not in
general real. Once again, the presence of $\F$ Gaussians is the
obstacle to writing down a general $\beta$ analogue.

We begin with a geometric, coordinate-free formulation.

%
\begin{theorem}\label{t.band_inv}
Let $T$ be a self-adjoint linear transformation on a finite-dimensional
inner product space $V$ of dimension $n$ over $\F$. An orthonormal
sequence $\{v_1,\ldots,v_r\}\subset V$ with $1\le r\le n$ can be
extended to an ordered \mbox{orthonormal} basis $\{v_1,\ldots,v_n\}$ for $V$
such that $ \langle v_i,Tv_j \rangle\ge0$ for $\llvert
i-j\rrvert=r$ and $ \langle v_i,Tv_j \rangle=0$ for
$\llvert i-j\rrvert>r$. Furthermore, if $ \langle
v_i,Tv_j \rangle> 0$ for $\llvert i-j\rrvert=r$ then
the extension is unique.
\end{theorem}

The point is that the same extension works for $T' = T+P$ provided
$P\in M_n(\F)$ satisfies $P\mid_{\{v_1,\ldots,v_r\}^\perp} = 0$.
In this case $\Span\{v_1,\ldots,v_r\}$ is also an invariant subspace
of $P$ and we speak of \textit{perturbing in this subspace}.

\begin{pf*}{Proof of Theorem~\ref{t.band_inv}}
We give an explicit inductive construction. Along the way,
we will see that the uniqueness condition holds precisely when the
choice is forced at each step.

It is convenient to restate the properties of the orthonormal basis in
the theorem in the following equivalent way: for $r+1\le i\le n$, we
have $ \langle v_{i},Tv_{i-r} \rangle\ge0$ and
$Tv_{i-r}\in\Span\{v_1,\ldots,v_i\}$. Suppose inductively that\break
$v_1,\ldots,v_{k-1}$ have been obtained for some $r+1\le k\le n$,
satisfying the preceding conditions for $r+1\le i\le k-1$. Let $w =
Tv_{k-r}$; we must choose $v_k$ so that $ \langle v_k,w
\rangle\ge0$ and $w\in\Span\{v_1,\ldots,v_k\}$. There are two
cases to consider. If $w\notin\Span\{v_1,\ldots,v_{k-1}\}$ then
$v_k$ must be a multiple of $w' = w-\sum_{i=1}^{k-1} \langle
v_i,w \rangle v_i$; the positivity condition further forces\vspace*{1pt} $v_k
= w'/\llvert w'\rrvert$, which gives $ \langle
v_k,w \rangle=\llvert w'\rrvert>0$. If $w\in\Span\{
v_1,\ldots,v_{k-1}\}$, then any $v_k\in\{v_1,\ldots,v_{k-1}\}^\perp
$ will do, and in this case $ \langle v_k,w \rangle=0$.
\end{pf*}

%
\begin{remark}When uniqueness holds, as is generically the case, the
basis may also be obtained by applying the Gram--Schmidt process to the
first $n$ vectors of the sequence
\[
v_1,\ldots,v_r,Tv_1,\ldots,Tv_r,T^2v_1,
\ldots,T^2v_r,\ldots.
\]
%
\end{remark}

We now state and prove a concrete matrix formulation in which the first
$r$ coordinate vectors play the role of $v_1,\ldots, v_r$. The point of
the second proof is that it emphasizes the resulting band matrix rather
than the change of basis; the algorithm will be used in the next subsection.

%
\begin{theorem}\label{t.band_mat}
Let $A\in M_n(\F)$ and $1\le r\le n$. There exists $U\in U_n(\F
)$ of the form $U=I_r\oplus\tilde U$ with $\tilde U\in U_{n-r}(\F)$
such that $B = UAU^\dag$ satisfies
%
%
\begin{eqnarray}
\label{B1} B_{ij}&\ge&0\qquad\mbox{for }1\le i,j\le n\mbox{ with }
\llvert i-j\rrvert=r,
\\
\label{B2} B_{ij}&=& 0\qquad\mbox{for }1\le i,j\le n\mbox{ with }
\llvert i-j\rrvert>r.
\end{eqnarray}
Furthermore, if strict positivity holds in (\ref{B1}) then $U$ and $B$
as such are unique.
\end{theorem}

We refer to $B$ as the \textit{band Jacobi form} of $A$. The allowed
perturbations here have the form $P = \tilde P\oplus0_{n-r}$ for
$\tilde P\in M_r(\F)$; these are invariant under conjugation by $U$,
so $U(A+P)U^\dag= B+P$.

\begin{pf*}{Proof of Theorem~\ref{t.band_mat}}
We prove existence by giving an explicit algorithm; it
generalizes the Lanczos algorithm, which applies in the case $r= 1$.
\begin{itemize}
%
\item For the first step, let $v = [A_{i,1}]_{r+1\le i\le n}\in\F
^{n-r}$ and take $\tilde U\in U_{n-r}(\F)$ such that $\tilde Uv =
|v|\tilde e_1$, where $\tilde e_1$ is the first standard basis vector
of $\F^{n-r}$. A concrete choice is the Househ\"older reflection $\tilde
U = I_{n-r}-2ww^\dag/w^\dag w$ with $w = v-|v|\tilde e_1$. Set $U_1 =
I_r\oplus\tilde U$ and $B_1 = U_1AU_1^\dag$.
%
\item Continue inductively: having obtained $U_{k},B_{k}$, let $v =
[(B_1)_{i,(k+1)}]_{r+k+1\le i\le n}\in\F^{n-r-k}$ and take $\tilde
U\in U_{n-r-k}(\F)$ such that $\tilde Uv = |v|\tilde e_1$. Set
$U_{k+1} = I_{r+k}\oplus\tilde U$ and $B_{k+1} =
U_{k+1}B_{k}U_{k+1}^\dag$.

\item Stop when $k = n-r$. Let $U = U_{n-r}\cdots U_1$ and $B = B_{n-r}
= UAU^\dag$.
\end{itemize}
It is immediate that $U$ and $B$ have the required properties. The
point is that the $k$th column of $B_k$ already ``looks right'', that
is, $(B_k)_{r+k,k}\ge0$ and $(B_k)_{r+l,k} = 0$ for $l> k$, and
subsequent transformations $U_{k+1},\ldots, U_{n-k}\in\{I_{r+k}\}
\oplus U_{n-r-k}(\F)$ ``don't mess it up''.

Toward uniqueness, suppose that $U',B'=U'A{U'}^\dag$ also have the
required properties and let $W = U'U^{-1}$ so that $B' = WBW^\dag$.
Assume inductively that $W\in\{I_{r+k}\}\oplus U_{n-r-k}(\F)$, which
is certainly true in the base case $k=0$. Write $W = I_{r+k}\oplus
\tilde W$. Let $b = [B_{i,k+1}]_{r+k+1\le i\le n}\in\F^{n-r-k}$ and
similarly for $b'$. Then $b' = \tilde Wb$. But $b = a\tilde e_1$ and
$b' = a'\tilde e_1$ with $a,a'>0$ by assumption. It follows that $a =
a'$ and $\tilde W\tilde e_1 = \tilde e_1$. Hence, $\tilde W \in\{I_1\}
\oplus U_{n-r-(k+1)}(\F)$ and $W \in\{I_{r+k+1}\}\oplus
U_{n-r-(k+1)}(\F)$, completing the induction step. We conclude that $W
= I_n$.
\end{pf*}

\subsection{Perturbed Gaussian and spiked Wishart models}

The change of basis described above interacts very nicely with the
Gaussian structure in Gaussian and Wishart random matrices. The $r=1$
case of this observation is due to \citet{Trotter}, who described
the tridiagonal forms explicitly. His forms fall into the framework of
Theorem~\ref{t.band_inv} by taking the initial vector to be fixed in
the Gaussian case, and taking it to be the top row of the data matrix
in the Wishart case. As we observed in Part~I, the change of
basis commutes with rank one additive perturbations for the Gaussian
case and with rank one spiking for the Wishart case. We now extend the
story to the $r>1$ setting.

In the Gaussian case, we will be perturbing in a fixed (nonrandom)
subspace; without loss of generality this may be taken as the initial
$r$-dimensional coordinate subspace, and so we take the basis of
Theorem~\ref{t.band_inv} that begins with the first $r$ standard basis
vectors. We can therefore obtain the band form by a direct application
of the algorithm from the proof of Theorem~\ref{t.band_mat}. The
Wishart case is a little more complicated; here we want to perturb in
the random subspace spanned by the first $r$ rows of the data matrix.
Our new basis will begin with the Gram--Schmidt orthogonalization of
these initial rows. As in the $r=1$ case, it is most transparent to
construct a lower band form of the data matrix first, afterward
realizing the band Jacobi form as its multiplicative symmetrization. In
both the Gaussian and the Wishart cases, we will see that the
uniqueness condition of Theorem~\ref{t.band_inv} holds almost surely.

Let $A$ be an $n\times n$ GOE matrix. Applying the algorithm from the
proof of Theorem~\ref{t.band_mat} while keeping track of the
distribution of the matrix $B_k$ at each step---the key of course being
the unitary invariance of standard Gaussian vectors---yields the
following band Jacobi random matrix $G = UAU^\dag$:
%
%
\begin{equation}
\label{G} G_{ij} = %
\cases{ \displaystyle\sqrt{\frac{2}\beta}
\wt g_i,& \quad$i=j$,
\vspace*{5pt}\vspace*{5pt}\cr
\displaystyle g_{ij}, &\quad$j<i<j+r$,
\vspace*{5pt}\cr
\displaystyle \frac{1}{\sqrt\beta}\chi_{(n-i+1)\beta},& \quad$i=j+r$,
\vspace*{5pt}\cr
0, &\quad$i>j+r$,
\vspace*{5pt}\cr
G_{ji}^*, &\quad$i<j$}
\end{equation}
for $1\le i,j\le n$, where the random variables appearing explicitly
are independent, $\wt g_i\sim N(0,1)$, $g_{ij}\sim\F N(0,1)$, and
$\chi_k\sim\operatorname{Chi}(k)$. The latter is the distribution of the
length of a $k$-dimensional standard Gaussian vector.

We can introduce a rank $r$ additive perturbation $A = A_0+\sqrt{n}P$,
where $P = \tilde P\oplus0_{n-r}$ with $\tilde P\in M_r(\F)$; since
$P$ commutes with the change of basis $U\in\{I_r\}\oplus U_{n-r}(\F
)$, we can write
%
%
\begin{equation}
\label{Gp} \qquad G = UAU^\dag= U(A_0+\sqrt{n}P)U^\dag=
UA_0U^\dag+ \sqrt{n}P = G_0 + \sqrt{n}P.
\end{equation}
As expected the perturbation shows up undisturbed in the upper-left
$r\times r$ corner of $G$.

Turning to the Wishart case, we first consider the null Wishart random
matrix $X^\dag X$, where $X$ is $p\times n$ with independent $\F N(0,1)$
entries. (Remember that $X^\dag X$ and $X X^\dag$ have the same nonzero
eigenvalues $\lambda_1,\ldots,\lambda_{n\wedge p}$.)
The final form can be described abstractly as given in the basis of
Theorem~\ref{t.band_inv} that extends the Gram--Schmidt
orthogonalization of the first $r$ rows of $X$. One cannot readily
obtain a description of the resulting random matrix from here, however,
so we give another way that generalizes Trotter's original procedure.
It is a ``singular value analogue'' of the algorithm from the proof of
Theorem~\ref{t.band_mat}, producing matrices $U\in U_n(\F)$ and $V\in
U_p(\F)$ such that $L = VXU$ has a ``lower band form'' that is zero
off the main and first $r$ sub-diagonals and positive on the outermost
of these. The key is to work alternately on rows and columns.
\begin{itemize}
\item Take $U_1\in U_n(\F)$ so that the first row of $XU_1$ lies in
the (positive) direction of the first coordinate basis vector of $\F^n$.
\item Take $V_1 = I_r\oplus U_{p-r}(\F)$ so that
$[(V_1XU_1)_{i,1}]_{r+1\le i\le p}\in\F^{p-r}$ lies in the direction
of the first coordinate basis vector of the latter subspace.
\item Take $U_2\in I_1\oplus U_{n-1}(\F)$ so that
$[(V_1XU_1U_2)_{2,j}]_{2\le j\le n}\in\F^{n-1}$ lies in the direction
of the first coordinate basis vector of the latter subspace.
\item Take $V_2\in I_{r+1}\oplus U_{p-r-1}(\F)$ so that
$[(V_2V_1XU_1U_2)_{i,2}]_{r+2\le j\le p}\in\F^{p-r-1}$ lies in the
direction of the first coordinate basis vector of the latter subspace.
\item Continue in this way until the rows and columns both run out
(stop alternating if one runs out before the other).
\end{itemize}
The resulting $L = V_{n\wedge(p-r)}\cdots V_1XU_1\cdots U_{n\wedge p}$
has $n\wedge p$ nonzero columns and $(n+r)\wedge p$ nonzero rows,
which can be described as follows:
%
%
\begin{equation}
\label{L} L_{ij} = %
\cases{ \displaystyle\frac{1}{\sqrt\beta}\wt
\chi_{(n-i+1)\beta}, &\quad$i=j$,
\vspace*{5pt}\cr
g_{ij},& \quad$j<i<j+r$,
\vspace*{5pt}\cr
\displaystyle\frac{1}{\sqrt\beta}\chi_{(p-i+1)\beta}, &\quad$i=j+r$,
\vspace*{5pt}\cr
0, &\quad$i<j$ or
$i>j+r$,}
\end{equation}
where\vspace*{1pt} the entries are independent, $\wt\chi_k,\chi_k\sim\operatorname{Chi}(k)$, $g_{ij}\sim\F N(0,1)$. Truncating the remaining zero rows
or columns, the matrix $S = L^\dag L$ is $(n\wedge p)\times(n\wedge p)$
and has the same \emph{nonzero} eigenvalues as $X^\dag X$. It has the
band form
%
%
\begin{equation}
\label{S}
\qquad S_{ij} = %
\cases{
\displaystyle\frac{1}{\beta}\wt
\chi_{(n-i+1)\beta}^2 + \sum_{i<k<i+r}
\llvert g_{k,i}\rrvert^2+\frac{1}\beta\chi
_{(p-i-r+1)\beta}^2,
\cr
\hspace*{124pt} i=j,
\vspace*{5pt}\cr
\displaystyle\frac{1}{\sqrt\beta}\wt\chi
_{(n-i+1)\beta}g_{ij}+\sum_{i<k<j+r}g_{k,i}^*g_{k,j}+
\frac{1}{\sqrt\beta}g^*_{j+r,i}\chi_{(p-j-r+1)\beta},
\cr
\hspace*{124pt} j<i<j+r,
\vspace*{5pt}\cr
\displaystyle\frac{1}\beta\wt\chi_{(n-i+1)\beta}\chi_{(p-i+1)\beta},\qquad i=j+r,
\vspace*{5pt}\cr
0,\hspace*{113pt} i>j+r,
\vspace*{5pt}\cr
S^*_{ji}, \hspace*{105pt} i<j,}
\end{equation}
where we have ignored the issue of truncation in the final $r$ rows and
columns ($g's$ and $\chi's$ with indices beyond the allowed range
should simply be zero). The change of basis is thus $U_1\cdots
U_{n\wedge p}$; a little thought shows that, as claimed earlier, the
new basis begins with the orthogonalization of the first $r$ rows of
$X$. Since the form~(\ref{S}) satisfies the uniqueness condition of
Theorem~\ref{t.band_inv} a.s., the basis is indeed the one given by
the theorem.

Now we consider the spiked Wishart matrix $X^\dag X = X_0^\dag\Sigma X_0$,
with $\Sigma= \tilde\Sigma\oplus I_{p-r}> 0$. Here $X_0$ is a null
Wishart matrix and $X = \Sigma^{1/2}X_0$. Notice that $X^\dag
X-X_0^\dag
X_0 = X_0^\dag((\tilde\Sigma-I_r)\oplus0 ) X_0$ is indeed
an additive perturbation in the subspace spanned by the first $r$ rows
of $X_0$. Since $\Sigma^{1/2}=\tilde\Sigma^{1/2}\oplus I$ commutes
with the inner transformation $V\in\{I_r\}\oplus U_{p-r}(\F)$, we have
%
\[
L^\dag L = U^\dag X^\dag X U = U^\dag
X_0^\dag\Sigma X_0 U = U^\dag
X_0^\dag V^\dag\Sigma VX_0U =
L_0^\dag\Sigma L_0,
\]
where $L = VXU$ and $L_0 = VX_0U$. The point is that same change of
basis works in the rank $r$ spiked case, and by the lower band
structure of $L_0$, the perturbation shows up in the upper-left
$r\times r$ corner:
%
%
\begin{equation}
\label{Sp} S - S_0 = L^\dag L-L_0^\dag
L_0 = \tilde L_0^\dag(\tilde\Sigma
-I_r)\tilde L_0\oplus0.
\end{equation}
Viewed in terms of the algorithm used to produce $L$, the point is that
the first $r$ rows of $X$ are never ``mixed'' together or with the
lower rows, but only ``rotated'' within themselves.

\section{Limits of block tridiagonal matrices}\label{s.limits}
The banded forms of Section~\ref{s.canonical} may also be considered
as block tridiagonal matrices with $r\times r$ blocks.
In this section, we give general conditions under which such random
matrices, appropriately scaled, converge at the soft spectral edge to a
random Schr\"odinger operator on the half-line with $r\times r$
matrix-valued potential and general self-adjoint boundary condition at
the origin. In Section~\ref{s.clt}, we verify these assumptions for
the two specific matrix models we consider.

Proposition~\ref{p.var} establishes that the limiting operator is
a.s. bounded below with purely discrete spectrum via a variational
principle. The main result is Theorem~\ref{t.conv}, which asserts that
the low-lying states of the discrete models converge to those of the
operator limit.

The scalar $r=1$ case of Part~I, based in turn on
{RRV}, serves as a prototype. Care is required throughout
to adapt the arguments to the matrix-valued setting, and we give a
self-contained treatment.

\subsection{Discrete model and embedding}

Underlying the convergence is the embedding of the discrete half-line
$\Z_+ = \{0,1,\ldots\}$ into the continuum $\R_+ = [0,\infty)$ via
$j\mapsto j/m_n$, where the scale factors $m_n\to\infty$ but with
$m_n = o(n)$. Define an associated embedding of vector-valued function
spaces by step functions:
\[
\ell^2_n\bigl(\Z_+,\F^r\bigr)\hookrightarrow
L^2\bigl(\R_+,\F^r\bigr),\qquad(v_0,v_1,
\ldots) \mapsto v(x) = v_{\lfloor m_n x\rfloor},
\]
which is isometric with $\ell^2_n$ norm $\llVert v\rrVert^2
= m_n^{-1}\sum_{j=0}^\infty\llvert v_j\rrvert^2$. (Recall
that $\F^r$ and $L^2$ have norms $\llvert v\rrvert^2 =
v^\dag
v$ and $\llVert f\rrVert^2 = \int_0^\infty\llvert
f\rrvert^2$, respectively.) Fix a standard basis for $\ell^2_n$
with lexicographic ordering
\[
(e_1,0,\ldots), (e_2,0,\ldots),\ldots,(e_r,0,
\ldots),(0,e_1,0,\ldots),\ldots,
\]
where $e_1,\ldots,e_r$ is the standard basis for $\F^r$. Identify $\F
^n$ with the $n$-dimen\-sional initial coordinate subspace of $\ell
^2_n$, consisting of $\F^r$-valued step-functions supported on the
interval $ [0,\lceil n/r\rceil/m_n )$ and with the final
step value in the subspace spanned by $e_1,\ldots,e_{r - (\lceil
n/r\rceil r-n)}$. Our $n\times n$ matrices will act on $\F^n$ with
respect to the above basis; we will generally assume the embedding $\F
^n\subset\ell^2_n\hookrightarrow L^2$ implicitly.

We define some operators on $L^2$, all of which leave $\ell^2_n$
invariant and may also be considered as infinite block matrices with
$r\times r$ blocks. The translation operator $(T_n f)(x) =
f(x+m_n^{-1})$ extends the left shift on $\ell^2_n$. Its adjoint
$T_n^\dag$ is the right shift, where $T_n^\dag f = 0$ on $[0,m_n^{-1})$.
The difference quotient $D_n = m_n(T_n - 1)$ extends a discrete
derivative. Write $\diag(A_0,A_1,\ldots)$ for both an $r\times r$
block diagonal matrix and its extension to a pointwise matrix
multiplication on $L^2$. Thus $E_n = \diag(m_nI_r,0,0,\ldots)$ is
scalar multiplication by $m_n\1_{[0,m_n^{-1})}$, a ``discretized delta
function at the origin''. Orthogonal projection from $\ell^2_n$ onto
$\F^n$ extends to a multiplication $R_n = \diag(I_r,\ldots,I_r,\break \diag
(1,\ldots,1,0,\ldots,0),0,\ldots)$, in which there are $\lceil
n/r\rceil$ nonzero blocks and a total of $n$ 1's.

Let $(Y_{n,i;j})_{j\in\Z_+}$, $i=1,2$ be two discrete-time $r\times
r$ matrix-valued random processes with
$Y_{n,1;j} \in M_r(\F)$ for all $j$. The processes may be embedded
into continuous time as above, by setting $Y_{n,i}(x) = Y_{n,i;\lfloor
m_n x\rfloor}$. Note also that $T_n$ and $\triangle_n =
m_n(1-T_n^\dag
)=-D_n^\dag$ may be sensibly applied to such matrix-valued functions.
The processes $Y_{n,i}$ are on- and off-diagonal integrated potentials,
and we define a ``potential operator'' by
%
%
\begin{equation}
\label{Vn} V_n = \diag(\triangle_n Y_{n,1})
+ \tfrac{1}{2} \bigl(\diag(\triangle_n Y_{n,2})T_n
+ T_n^\dag\diag\bigl(\triangle_n
Y_{n,2}^\dag\bigr) \bigr).
\end{equation}
Fix $W_n\in M_r(\F)$, a nonrandom ``boundary term''.

Finally, consider
%
%
\begin{equation}
\label{Hn} H_n = R_n \bigl(D_n^\dag
D_n + V_n + W_n E_n
\bigr)R_n.
\end{equation}
This operator leaves the initial coordinate subspace $\F^n$ invariant;
we shall also use $H_n$ to denote the \emph{matrix of its restriction
to} $\F^n$. The matrix $H_n\in M_n(\F)$ is self-adjoint and block
tridiagonal up to a truncation in the lower-right corner. Its main-
and super-diagonal processes are
%
%
\begin{eqnarray}\label{Hn12}
&& m_n^2+m_n(W_n +
Y_{n,1;0}),2m_n^2+m_n(Y_{n,1;1}-Y_{n,1;0}),
\nonumber
\\
&&\qquad  {}2m_n^2 +m_n(Y_{n,1;2}-Y_{n,1;1}),
\ldots
\\
&&\qquad{} -m_n^2 +\tfrac{1}{2}m_nY_{n,2;0},
-m_n^2+\tfrac
{1}{2}m_n(Y_{n,2;1}-Y_{n,2;0}),
\ldots,\nonumber
\end{eqnarray}
respectively; the sub-diagonal process is of course the conjugate
transpose of the super-diagonal process. (We could have absorbed $W_n$
into $Y_{n,1}$ as an additive constant, but keep it separate for
reasons that will soon be clear. Note also that the upper-left block
has $m_n^2$ rather than $2m_n^2$.) We refer to $H_n$ as a \textit{rank
$r$ block tri-diagonal ensemble}.

As in {RRV} and Part~I, convergence rests on a
few key assumptions on the potential and boundary terms just
introduced. By choice, no additional scaling will be required. The role
of the convergence in the first and third assumption below will be
clear as soon as we define the continuum limit. The growth and
oscillation bounds of the second assumption (and the lower bound
implied by the third) ensure tightness of the low-lying states; in
particular, they guarantee that the spectrum remains discrete and
bounded below in the limit.

\begin{ass}[(Tightness and convergence)]\label{ass1}
There exists a continuous $M_r(\F)$-valued random
process $\{Y(x)\}_{x\ge0}$ with $Y(0)=0$ such that
%
%
\begin{eqnarray}
\label{a1} \bigl\{Y_{n,i}(x)\bigr\}_{x\ge0},\qquad i=1,2
\mbox{ are tight in law,}
\nonumber\\[-8pt]\\[-8pt]\nonumber
Y_{n,1}+\tfrac{1}{2}\bigl(Y_{n,2}+Y_{n,2}^\dag
\bigr) \Rightarrow Y\qquad\mbox{in law}
\end{eqnarray}
with respect to the compact-uniform topology (defined using any matrix norm).
\end{ass}

\begin{ass}[(Growth and oscillation bounds)]\label{ass2}
There is a decomposition
%
%
\begin{equation}
\label{Ydec} Y_{n,i;j} = m_n^{-1}\sum
_{k=0}^{j}\eta_{n,i;k}+\omega_{n,i;j}
\end{equation}
(so $\triangle_n Y_{n,i} = \eta_{n,i} + \triangle_n\omega_{n,i}$)
with $\eta_{n,i;j}\ge0$ (as matrices), such that for some
deterministic scalar continuous nondecreasing unbounded functions $\ol
\eta(x)>0$, $\zeta(x)\ge1$ not depending on $n$, and random
constants $\kappa_n\ge1$ defined on the same probability spaces, the
following hold: the $\kappa_n$ are tight in distribution, and for each
$n$ we have almost surely
%
%
\begin{eqnarray}
\label{a21} \ol\eta(x)/\kappa_n-\kappa_n \le
\eta_{n,1}(x)+\eta_{n,2}(x) &\le&\kappa_n \bigl(1+
\ol\eta(x) \bigr),
\\
\label{a22} \eta_{n,2}(x) &\le&2m_n^2,
\\
\label{a23} \bigl\llvert\omega_{n,1}(\xi)-\omega_{n,1}(x)
\bigr\rrvert^2 +\bigl\llvert\omega_{n,2}(\xi)-
\omega_{n,2}(x)\bigr\rrvert^2 &\le&\kappa_n
\bigl(1+\ol\eta(x)/\zeta(x) \bigr)
\end{eqnarray}
for all $x,\xi\in[0,\lceil n/r\rceil/m_n)$ with $\llvert\xi
-x\rrvert\le1$. Here, matrix inequalities have their usual
meaning and
single bars denote the spectral [or $\ell^2(\F^r)$ operator] norm.
\end{ass}

\begin{ass}[(Critical or subcritical perturbation)]\label{ass3}
For some orthonormal basis $u_1,\ldots
,u_r$ of $\F^r$ and $-\infty<w_1\le\cdots\le w_r\le\infty$ we
have $W_n = \sum_{i=1}^r w_{n,i}u_i u_i^\dag$, where $w_{n,i}\in\R$
satisfy $\lim_{n\to\infty}w_{n,i}=w_i$ for each $i$.
\end{ass}

We write $r_0 = \#\{i:w_i<\infty\}\in\{0,\ldots,r\}$ for the
``critical rank''. Formally, $W_n\to W = \sum_{i=1}^r w_i u_i u_i^\dag
\in M_r^*(\F)$. It is natural to view $W$ as a parameter: that is, we
will consider the joint behaviour of the model (for given $Y_{n,i},
Y$) over all $W_n, W$ satisfying Assumption~\ref{ass3}.

\subsection{Reduction to deterministic setting}

In the next subsection, we will define a limiting object in terms of
$Y(x)$ and $W$; we want to prove that the discrete models converge to
this continuum limit in law. We reduce the problem to a deterministic
convergence statement as follows. First, select any subsequence. It
will be convenient to extract a further subsequence so that certain
additional tight sequences converge jointly in law; Skorokhod's
representation theorem [see \citet{EthierKurtz}] says this
convergence can be realized almost surely on a single probability
space. We may then proceed pathwise.

In detail, consider (\ref{a1})--(\ref{a23}). Note in particular that
nonnegativity of the $\eta_{n,i}$ and the upper bound of~(\ref{a21})
give that for $i=1,2$ the piecewise linear process $ \{\int_0^x\eta
_{n,i} \}_{x\ge0}$ is tight in distribution, pointwise
with respect to the spectral norm and in fact compact-uniformly. Given
a subsequence, we pass to a further subsequence so that the following
distributional limits exist jointly:
%
%
\begin{eqnarray}\label{ar}
Y_{n,i} &\Rightarrow&
Y_i,\nonumber
\\
{\int_0\eta_{n,i}} &\Rightarrow& \wt
\eta_i,
\\
\kappa_n &\Rightarrow&\kappa,\nonumber 
\end{eqnarray}
for $i=1,2$, where convergence in the first two lines is in the
compact-uniform topology. We realize~(\ref{ar})
pathwise a.s. on some probability space and continue in this
deterministic setting.

We can take (\ref{a21})--(\ref{a23}) to hold with $\kappa_n$
replaced with a single $\kappa$.
Observe that~(\ref{a21}) gives a local Lipschitz bound on the $\int
\eta_{n,i}$, which is inherited by their limits $\wt\eta_i$ (the
spectral norm controls the matrix entries). Thus, $\eta_i = \wt\eta
_i{}'$ is defined almost everywhere on $\R_+$, satisfies~(\ref{a21}),
and may be defined to satisfy this inequality everywhere. Furthermore,
one easily checks that $m_n^{-1}\sum\eta_{n,i}\to\int\eta_i$
compact-uniformly as well (use continuity of the limit). Therefore,
$\omega_{n,i} = y_{n,i}-m_n^{-1}\sum\eta_{n,i}$ must have a
continuous limit $\omega_i$ for $i = 1,2$; moreover, the bound~(\ref
{a23}) is inherited by the limits. Lastly, put $\eta= \eta_1+\eta_2$,
$\omega= \omega_1+ \frac{1}{2}(\omega_2+\omega_2^\dag)$ and note that
$Y_i = \int\eta_i + \omega_i$ and $Y = \int\eta+ \omega$. For
convenience, we record the bounds inherited by $\eta, \omega$:
%
%
\begin{eqnarray}
\label{ac1} \ol\eta(x)/\kappa-\kappa\le\eta(x) &\le&\kappa\bigl(1+\ol
\eta(x) \bigr),
\\
\label{ac2} \bigl\llvert\omega(\xi)-\omega(x)\bigr\rrvert^2 &\le&
\kappa\bigl(1+\ol\eta(x)/\zeta(x) \bigr) 
\end{eqnarray}
for $x,\xi\in\R_+$ with $\llvert\xi-x\rrvert\le1$ (and
note that $\kappa\ge1$).

We will assume this \textit{subsequential pathwise coupling} for the
remainder of the section.

\subsection{Limiting object and variational characterization}

Formally, the limiting object is the eigenvalue problem
%
%
\begin{eqnarray}
\label{e_prob} \Hc f &=& \Lambda f\qquad\mbox{on }L^2\bigl(\R_+,
\F^r\bigr),
\nonumber\\[-8pt]\\[-8pt]\nonumber
f'(0) = Wf(0),
\end{eqnarray}
where
\[
\Hc= -\frac{d^2}{dx^2} + Y'(x).
\]
Writing the spectral decomposition $W =\sum_{i=1}^{r} w_i u_i u_i^\dag
$, recall (Assumption~\ref{ass3}) that we actually allow $w_i\in\R$ for $1\le
i\le r_0$ and, symbolically, $w_i = +\infty$ for $r_0+1\le i\le r$.
Writing $f_i = u_i^\dag f$, the boundary condition is then to be
interpreted as
%
%
\begin{equation}
\label{Wrig} %
\begin{aligned} f'_i(0) &=&
w_i f_i(0)\qquad\mbox{for }i\le r_0,
\\
f_i(0) &=& 0 \qquad\mbox{for }i>r_0.
\end{aligned}
\end{equation}
We thus have a completely general homogeneous linear self-adjoint
boundary condition. We refer to $\Span\{u_i:i>r_0\}$ as the \textit
{Dirichlet subspace} and the corresponding $f_i$ as \textit{Dirichlet
components}; they will require special treatment in what follows.

We will actually work with a symmetric bilinear form (properly,
sesquilinear if $\F= \C$ or $\HH$) associated with the eigenvalue
problem~(\ref{e_prob}). Define a space of test functions $C_0^\infty$
consisting of smooth $\F^r$-valued functions $\ph$ on $\R_+$ with
compact support; we additionally require the Dirichlet components to be
supported away from the origin. Introduce a symmetric bilinear form on
$C_0^\infty\times C_0^\infty$ by
%
%
\begin{equation}
\label{Hc} \Hc(\ph,\psi) = \bigl\langle\ph',\psi'
\bigr\rangle- \bigl\langle\ph',Y\psi\bigr\rangle- \bigl\langle
\ph,Y\psi' \bigr\rangle+ \ph(0)^\dag W\psi(0),
\end{equation}
where the Dirichlet part of the last term is interpreted as zero.
Formally, the form $\Hc(\cdot,\cdot)$ is just the usual one $
\langle\cdot,\Hc\cdot\rangle$ associated with the operator
$\Hc$; the potential term has been integrated by parts and the
boundary condition ``built in''. See also Remark~\ref{r.ee} below.

The regularity and decay conditions naturally associated with this form
are given by the following weighted Sobolev norm:
%
%
\begin{equation}
\label{Lstar} \llVert f\rrVert_*^2 = \int_0^\infty
\bigl(\bigl\llvert f'\bigr\rrvert^2 + (1+\ol\eta)
\llvert f\rrvert^2 \bigr) + f(0)^\dag W^+ f(0),
\end{equation}
where the \textit{positive part of} $W$ is defined as $W^+ = \sum
_{i=1}^{r} w_i^+ u_i u_i^\dag$ with $w^+ = w\vee0$. [Define the
\textit
{negative part} similarly with $w_i^- = -(w\wedge0)$, so that $W = W^+
- W^-$.] We refer to $\llVert\cdot\rrVert_*$ as the $L^*$
\textit{norm} and define an associated Hilbert space $L^*$ as the
closure of $C_0^\infty$ under this norm. (The formal Dirichlet terms
are again interpreted to be zero, but they can also be thought of as
imposing the Dirichlet condition.) We record some basic facts about $L^*$.

%
\begin{fact}\label{f.1}
Any $f\in L^*$ is uniformly $\mbox{H\"older}(1/2)$-continuous and satisfies
$\llvert f(x)\rrvert^2\le2\llVert f'\rrVert\llVert f\rrVert$ $\le
\llVert f\rrVert_*^2$ for all $x$;
furthermore, $f_i(0) = 0$ for $i>r_0$.
\end{fact}

\begin{pf}
We have $\llvert f(y)-f(x)\rrvert= \llvert\int_x^y
f'\rrvert\le\llVert f'\rrVert\llvert y-x\rrvert^{1/2}$. For $f\in
C_0^\infty$ we have $\llvert f(x)\rrvert^2 = -\int_x^\infty2\Re
f^\dag f' \le2\llVert f\rrVert
\llVert f'\rrVert\le\llVert f\rrVert_*^2$; an
$L^*$-bounded sequence in $C_0^\infty$, therefore, has a
compact-uniformly convergent subsequence, so we can extend this bound
to $f\in L^*$ and also conclude the behaviour in the Dirichlet components.
\end{pf}

%
\begin{fact}\label{f.2} Every $L^*$-bounded sequence has a subsequence
converging in the following modes: \textup{(i)} weakly in $L^*$, \textup{(ii)}
derivatives weakly in $L^2$, \textup{(iii)} uniformly on compacts and \textup{(iv)} in $L^2$.
\end{fact}

\begin{pf} (i) and (ii) are just Banach--Alaoglu; (iii) is the
previous fact and Arzel\`a--Ascoli again; (iii) implies $L^2$
convergence locally, while the uniform bound on $\int\ol\eta\llvert
f_n\rrvert^2$ produces the uniform integrability required
for (iv). Note that the weak limit in (ii) really is the derivative of
the limit function, as one can see by integrating against functions $\1
_{[0,x]}$ and using pointwise convergence.
\end{pf}

By the bound in Fact~\ref{f.1} with $x=0$, the boundary term in (\ref
{Lstar}) could be done away with. It is natural to include the term,
however, when considering all $W$ simultaneously and viewing the
Dirichlet case as a limiting case. More importantly, it clarifies the
role of the boundary terms in the following key bound.

%
\begin{lemma}\label{l.cbound} For every $0<c<1/\kappa$ there is a
$C>0$ such that, for each $b> 0$, the following holds for all $W\ge-b$
and all $f\in C_0^\infty$:
%
%
\begin{equation}
\label{cbound} c\llVert f\rrVert_*^2 - \bigl(1+b^2
\bigr)C\llVert f\rrVert^2 \le\Hc(f,f) \le C\llVert f\rrVert
_*^2.
\end{equation}
In particular, $\Hc(\cdot,\cdot)$ extends uniquely to a continuous
symmetric bilinear form on $L^*\times L^*$.
\end{lemma}

\begin{pf}
For the first three terms of~(\ref{Hc}), we use the decomposition
$Y=\int\eta+\omega$ from the previous subsection. Integrating the
$\int\eta$ term by parts, (\ref{ac1}) easily yields
\[
\frac{1}{\kappa}\llVert f\rrVert_*^2 - \kappa\llVert f\rrVert
^2 \le\bigl\llVert f'\bigr\rrVert^2 +
\langle f,\eta f \rangle\le\kappa\llVert f\rrVert_*^2.
\]
Break up the $\omega$ term as follows. The moving average $\ol\omega
_x = \int_x^{x+1}\omega$ is differentiable with $\ol\omega_x' =
\omega_{x+1}-\omega_{x}$; writing $\omega= \ol\omega+ (\omega-
\ol\omega)$, we have
\[
-2\Re\bigl\langle f',\omega f \bigr\rangle= \bigl\langle f,\ol
\omega'f \bigr\rangle+ 2\Re\bigl\langle f',(\ol
\omega-\omega)f \bigr\rangle.
\]
By~(\ref{ac2}), $\max(\llvert\omega_{\xi}-\omega
_{x}\rrvert,\llvert\omega_{\xi}-\omega_x\rrvert
^2 )\le C_\varepsilon+\varepsilon\ol\eta(x)$ for $\llvert
\xi-x\rrvert\le1$, where $\varepsilon$ can be made small. In
particular, the first term above is bounded absolutely by $\varepsilon
\llVert f\rrVert_*^2+C_\varepsilon\llVert f\rrVert
^2$. Averaging, we also get $\llvert\ol\omega_x - \omega
_x\rrvert\le(C_\varepsilon+\varepsilon\ol\eta(x))^{1/2}$;\vspace*{1pt}
Cauchy--Schwarz then bounds the second term absolutely by $\sqrt
{\varepsilon}\int_0^\infty\llvert f'\rrvert^2+\frac
{1}{\sqrt{\varepsilon}}\int_0^\infty{(C_\varepsilon+\varepsilon
\ol\eta)\llvert f\rrvert^2}$ and thus by $\sqrt
{\varepsilon}\llVert f\rrVert_*^2+C_\varepsilon'\llVert
f\rrVert^2$. Now combine all the terms and set $\varepsilon$
small to obtain a version of~(\ref{cbound}) with the boundary terms
omitted (from both the form and the norm).

We break the boundary term in~(\ref{Hc}) into its positive and
negative parts. For the negative part, Fact~\ref{f.2} gives $\llvert
f(0)\rrvert^2 \le(\varepsilon/b)\llVert f'\rrVert^2+ (b/\varepsilon
)\llVert f\rrVert^2$; $W^-\le b$
then implies that
\[
0 \le f(0)^\dag W^- f(0) \le\varepsilon\llVert f\rrVert
_*^2+C''_\varepsilon b^2
\llVert f\rrVert^2,
\]
which may be subtracted from the inequality already obtained. For the
positive part $f(0)^\dag W^+ f(0)$, use the fact that $c\le1\le C$ to
simply add it in. We thus arrive at~(\ref{cbound}).

For the $L^*$ bilinear form bound, begin with the quadratic form bound
$\llvert\Hc(f,f)\rrvert\le C_{c,b}\llVert f\rrVert
^2_*$; it is a standard Hilbert space fact that it may be polarized to
a bilinear form bound [see, e.g., Section~18 of \citet{Halmos}].
\end{pf}

%
\begin{definition}\label{d.ee} We say $f\in L^*$ is an \textit
{eigenfunction} with \textit{eigenvalue} $\Lambda$ if $f\neq0$ and
for all $\ph\in C_0^\infty$ we have
%
%
\begin{equation}
\label{ee} \Hc(\ph,f) = \Lambda\langle\ph,f \rangle.
\end{equation}
Note that~(\ref{ee}) then automatically holds for all $\ph\in L^*$,
by $L^*$-continuity of both sides.
\end{definition}

%
\begin{remark}\label{r.ee} This definition represents a weak or
distributional version of the problem~(\ref{e_prob}). As further
justification, integrate by parts to write the definition
\[
\bigl\langle\ph',f' \bigr\rangle- \bigl\langle
\ph',Yf \bigr\rangle- \bigl\langle\ph,Yf' \bigr
\rangle+ \ph(0)^\dag Wf(0) = \Lambda\langle\ph,f \rangle
\]
in the form
\[
\bigl\langle\ph',f' \bigr\rangle- \bigl\langle
\ph',Yf \bigr\rangle+ \biggl\langle\ph',{\int
_0 Yf'} \biggr\rangle- \bigl\langle
\ph',Wf(0) \bigr\rangle= -\Lambda\biggl\langle\ph',{
\int_0 f} \biggr\rangle,
\]
which is equivalent to
%
%
\begin{equation}
\label{integrated} f'(x) = Wf(0) + Y(x)f(x) - \int_0^x
Yf' - \Lambda\int_0^x f\qquad
\mbox{a.e. }x.
\end{equation}
(For a Dirichlet component $f_i$ the restriction on test functions
implies that $ \langle\ph'_i,1 \rangle= 0$, so the first
boundary term on the right-hand side is replaced with an arbitrary
constant.) Now (\ref{integrated}) shows that $f'$ has a continuous
version, and the equation may be taken to hold everywhere. In
particular, $f$ satisfies the boundary condition of~(\ref{e_prob})
classically. [For a Dirichlet component, we just find that the
arbitrary constant is $f_i'(0)$.] One can also view (\ref{integrated})
as a straightforward integrated version of the eigenvalue equation in
which the potential term has been interpreted via integration by parts.
This equation will be useful in Lemma~\ref{l.one} below and is the
starting point for the development in Section~\ref{s.alternative}.
\end{remark}

We now characterize the eigenvalues and eigenfunctions variationally.
As usual, it follows from the symmetry of the form that eigenvalues are
real (and eigenfunctions with distinct eigenvalues are
$L^2$-orthogonal). The $L^2$ part of the lower bound in (\ref{cbound})
says the spectrum is bounded below. The rest of (\ref{cbound}) implies
that there are only finitely many eigenvalues below any given level: a
sequence of normalized eigenfunctions with bounded eigenvalues must
have an $L^2$-convergent subsequence by Fact \ref{f.2}. At a given
level, more is true.

%
\begin{lemma}\label{l.one} For each $\Lambda\in\R$, the
corresponding eigenspace is at most $r$-dimensional.
\end{lemma}

\begin{pf}
By linearity, it suffices to show a solution of~(\ref{integrated})
with $f'(0)=f(0)=0$ must vanish identically. Integrate by parts to write
\[
f'(x) = Y(x) \int_0^x
f' - \int_0^x Yf' -
\Lambda x \int_0^x f' +\Lambda
\int_0^x tf'(t)\,dt,
\]
which implies that $\llvert f'(x)\rrvert\le C(x)\int_0^x\llvert
f'\rrvert$ with some $C(x)<\infty$ increasing in
$x$. Gronwall's lemma then gives $\llvert f'(x)\rrvert= 0$
for all $x\ge0$.
\end{pf}

%
\begin{proposition}\label{p.var} There is a well-defined $(k+1)$st
lowest eigen\-value $\Lambda_k$, counting with multiplicity. The
eigenvalues $\Lambda_0\le\Lambda_1\le\ldots$ together with an
orthonormal sequence of corresponding eigenvectors $f_0,f_1,\ldots$
are given recursively by the variational problem
\[
\Lambda_k = \mathop{\inf_{f\in L^*, \llVert f\rrVert
=1,}}_{{f\perp f_0,\ldots, f_{k-1}}}
\Hc(f,f)
\]
in which the minimum is attained and we set $f_k$ to be any minimizer.
\end{proposition}

%
\begin{remark} Since we must have $\Lambda_k\to\infty$, $\{\Lambda
_0,\Lambda_1,\ldots\}$ exhausts the spectrum and the resolvent
operator is compact. We do not make this statement precise.
\end{remark}

\begin{pf} First taking $k = 0$, the infimum $\tilde\Lambda$ is
finite by~(\ref{cbound}). Let $f_n$ be a minimizing sequence; it is
$L^*$-bounded, again by~(\ref{cbound}). Pass to a subsequence
converging to $f\in L^*$ in all the modes of Fact~\ref{f.2}. In
particular, $1 = \llVert f_n\rrVert\to\llVert f\rrVert$, so $\Hc
(f,f)\ge\tilde\Lambda$ by definition. But also
\begin{eqnarray*}
\Hc(f,f) &=& \bigl\llVert f'\bigr\rrVert^2 + \langle
f,\eta f \rangle+ \bigl\langle f,\ol\omega'f \bigr\rangle+ 2\Re
\bigl\langle f',(\ol\omega-\omega)f \bigr\rangle+
f(0)^\dag W f(0)
\\
&\le&\liminf_{n\to\infty} \Hc(f_n,f_n)
\end{eqnarray*}
by a term-by-term comparison. Indeed, the inequality holds for the
first term by weak convergence, and for the second term by pointwise
convergence and Fatou's lemma; the remaining terms are just equal to
the corresponding limits, because the second members of the inner
products converge in $L^2$ by the bounds from the proof of Lemma~\ref
{l.cbound} together with $L^*$-boundedness and $L^2$-convergence.
Therefore, $\Hc(f,f) = \tilde\Lambda$.

A standard argument now shows $(\tilde\Lambda,f)$ is an
eigenvalue--eigenfunction pair: taking $\ph\in C_0^\infty$ and
$\varepsilon$ small, put $f^\varepsilon= (f+\varepsilon\ph)/{\llVert
f+\varepsilon\ph\rrVert}$; since $f$ is a minimizer,
$\frac{d}{d\varepsilon}\mid_{\varepsilon=0}\Hc
(f^\varepsilon,f^\varepsilon)$ must vanish;
the latter says precisely~(\ref{ee}) with $\tilde\Lambda$. Finally,
suppose $(\Lambda,g)$ is any eigenvalue--eigenfunction\vspace*{1pt} pair; then $\Hc
(g,g) = \Lambda$, and hence $\tilde\Lambda\le\Lambda$. We are thus
justified in setting $\Lambda_0 = \tilde\Lambda$ and $f_0=f$.

Proceed inductively, minimizing now over the orthocomplement $\{f\in
L^*: \llVert f\rrVert=1, f\perp f_0,\ldots, f_{k-1}\}$.
Again, $L^2$-convergence of a minimizing sequence guarantees that the
limit remains admissible; as before, the limit is in fact a minimizer;
conclude by applying the arguments of the previous paragraph with $\ph
,g$ also restricted to the orthocomplement.
\end{pf}

\subsection{Statement}

We are finally ready to state the main result of this section. Recall
that we consider eigenvectors of a matrix $H_n\in M_n(\F)$ in the
embedding $\F^n\subset\ell^2_n(\Z_+,\F^r)\hookrightarrow L^2(\R
_+,\F^r)$ above.

%
\begin{theorem}\label{t.conv} Let $H_n$ be a rank $r$ block
tr-diagonal ensemble as in~(\ref{Hn}) satisfying Assumptions \ref{ass1}--\ref{ass3}, and
let $\lambda_{n,k}$ be its $(k+1)$st lowest eigenvalue. Define the
associated form $\Hc$ as in~(\ref{Hc}) and let $\Lambda_k$ be its
a.s. defined $(k+1)$st lowest eigenvalue. In the deterministic setting
of subsequential pathwise coupling, $\lambda_{n,k}\to\Lambda_k$ for
each $k = 0,1,\ldots.$ Furthermore, a sequence of normalized
eigenvectors corresponding to $\lambda_{n,k}$ is precompact in $L^2$
norm, and every subsequential limit is an eigenfunction corresponding
to $\Lambda_k$. Finally, convergence holds uniformly over possible
$W_n,W\ge-b>-\infty$. One recovers the corresponding distributional
tightness and convergence statements for the full sequence, jointly for
$k = 0,1,\ldots$ in the sense of finite-dimensional distributions and
jointly over $W_n,W$.
\end{theorem}

%
\begin{remark} The eigenvector convergence statement requires
subsequences for two reasons: possible multiplicity of the limiting
eigenvalues, and the sign or phase ambiguity of the eigenvectors. It is
possible to formulate the conclusion of the theorem very simply using
spectral projections. [If $H$ has purely discrete spectrum, the
spectral projection $\1_A(H)$ is simply orthogonal projection of $L^2$
onto the span of those eigenvectors of $H$ whose eigenvalues lie in
$A\subset\R$.] The joint eigenvalue-eigenvector convergence may be
restated in the deterministic setting as follows: \emph{For all $a\in
\R\setminus\{\Lambda_0,\Lambda_1,\ldots\}$, the spectral
projections $\1_{(-\infty,a)}(H_n)\to\1_{(-\infty,a)}(\Hc)$ in
$L^2$ operator norm.} The corresponding distributional statement holds
jointly over all $a$ that are a.s. off the limiting spectrum (or
simply all $a$ if the distributions of the $\Lambda_k$ are nonatomic).
\end{remark}

%
\begin{remark}
An operator-theoretic formulation of the theorem (which we do not
develop here) would state a norm resolvent convergence: the resolvent
matrices, precomposed with the finite-rank projections $L^2\to\F^n$
associated with the embedding, converge to the continuum resolvent in
$L^2$ operator norm. This mode of convergence is the strongest one can
hope for in the unbounded setting [see, e.g., Section VIII.7 of \citet{RS1,Weidmann}].
\end{remark}

The proof will be given over the course of the next two subsections.

\subsection{Tightness}

We will need a discrete analogue of the $L^*$ norm and a counterpart of
Lemma~\ref{l.cbound} with constants uniform in $n$. For $v\in\F
^n\hookrightarrow L^2(\R_+,\F^r)$ as above, define the \textit
{$L^*_n$ norm} by
%
%
\begin{eqnarray}
\label{Lstar_n} \llVert v\rrVert_{*n}^2 &=& \bigl\langle v,
\bigl(D_n^\dag D_n+ 1+\ol\eta+E_n
W_n^+ \bigr)v \bigr\rangle
\nonumber\\[-8pt]\\[-8pt]\nonumber
&=& \int_0^\infty\bigl(\llvert D_n v
\rrvert^2 + (1+\ol\eta)\llvert v\rrvert^2
\bigr)+v(0)^\dag W_n^+ v(0)
\end{eqnarray}
with the nonnegative part $W_n^+$ defined as before.

%
\begin{remark}\label{r.dir} When considering just a single $W_n,W$,
the boundary term in~(\ref{Lstar_n}) is really only required when the
limit includes Dirichlet terms; it is simpler, however, not to
distinguish the two cases here. More importantly, including this term
clarifies the role of the boundary term in the following key bound.
Note that the original case considered in {RRV} has $W_n =
m_n$ in our notation. (The $H_n$ form and $L^*_n$ norm there contained
a term $m_n\llvert v_0\rrvert^2$, though it is hidden in the
fact that, in our notation, they use $\triangle_n$ in place of $D_n$.)
\end{remark}

%
\begin{lemma}\label{l.dbound} For every $0<c<1/4\kappa$ there is a
$C>0$ such that, for each $b>0$, the following holds for all $n$,
$W_n\ge-b$ and $v\in\F^n$:
%
%
\begin{equation}
\label{dbound} c\llVert v\rrVert_{*n}^2-
\bigl(1+b^2\bigr)C\llVert v\rrVert^2 \le\langle
v,H_n v \rangle\le C\llVert v\rrVert_{*n}^2.
\end{equation}
\end{lemma}

\begin{pf} We drop the subscript $n$. The form associated with
(\ref{Hn}) is
%
%
\begin{equation}
\label{Hnform} \langle v,H v \rangle= \llVert D v\rrVert^2 + \langle
v,V v \rangle+v(0)^\dag W v(0).
\end{equation}
The potential term $ \langle v,Vv \rangle= \int_0^{\infty
} v^\dag V v$, defined in (\ref{Vn}), is analyzed according to~(\ref{Ydec}):
\begin{eqnarray*}
v^\dag Vv &=& v^\dag(\triangle Y_1)v + \Re
v^\dag(\triangle Y_2)Tv
\\
&=& \bigl(v^\dag\eta_1 v +\Re v^\dag
\eta_2Tv \bigr) + \bigl(v^\dag(\triangle
\omega_1)v+\Re v^\dag(\triangle\omega_2)Tv
\bigr).
\end{eqnarray*}
Together with $\llvert D_n v\rrvert^2$, the $\eta$-terms
provide the structure of the bound as we now show. Afterward we will
control the $\omega$-terms and lastly deal with the boundary term.

Recall~(\ref{a21}) and that $\eta_i\ge0$. For an upper bound,
rearrange $(v-Tv)^\dag\eta_2(v-Tv)\ge0$ to
\begin{eqnarray*}
\Re v^\dag\eta_2 Tv &\le&\tfrac{1}{2}v^\dag
\eta_2v+\tfrac
{1}{2}(Tv)^\dag\eta_2 Tv
\\
&\le&\tfrac{1}{2}\kappa(\ol\eta+1) \bigl(\llvert v\rrvert^2+
\llvert Tv\rrvert^2 \bigr). 
\end{eqnarray*}
Now $\int\ol\eta\llvert Tv\rrvert^2=\int(T^\dag\ol\eta
)\llvert v\rrvert^2\le\int\ol\eta\llvert v\rrvert^2$ since $\ol\eta
$ is nondecreasing, and we obtain
%
%
\begin{equation}
\label{eta_u} \llVert Dv\rrVert^2 + \langle v,\eta_1 v
\rangle+\Re\langle v,\eta_2 Tv \rangle\le2\kappa\llVert v\rrVert
_*^2.
\end{equation}
Toward a lower bound, we use the slightly tricky rearrangement $0\le
(\frac{1}{2}v+Tv)^\dag\eta_2(\frac{1}{2}v+Tv) = 3\Re v^\dag\eta
_2 Tv+ (Tv-v)^\dag
\eta_2(Tv-v)-\frac{3}4 v^\dag\eta_2 v$. With~(\ref{a22}), we get
\begin{eqnarray*}
\Re v^\dag\eta_2 Tv &\ge&-\tfrac{1}3(Tv-v)^\dag
\eta_2(Tv-v)+\tfrac{1}4 v^\dag\eta_2 v
\\
&\ge&-\tfrac{2}3\llvert Dv\rrvert^2+\tfrac{1}4
v^\dag\eta_2 v,
\end{eqnarray*}
so by~(\ref{a21}),
\[
\llvert Dv\rrvert^2 + v^\dag\eta_1 v+ \Re
v^\dag\eta_2 Tv \ge\tfrac{1}3\llvert Dv\rrvert
^2 +\tfrac{1}4(\ol\eta/\kappa-\kappa)\llvert v\rrvert
^2
\]
and thus
%
%
\begin{equation}
\label{eta_l} \llVert Dv\rrVert^2 + \langle v,\eta_1 v
\rangle+\Re\langle v,\eta_2 Tv \rangle\ge(1/4\kappa)\llVert v\rrVert
^2_* -(\kappa/4)\llVert v\rrVert^2.
\end{equation}

We handle the $\omega$-terms with a discrete analogue of the
decomposition used in the continuum proof. Consider the moving average
\[
\ol\omega_i = \lfloor m\rfloor^{-1}\sum
_{j=1}^{\lfloor m\rfloor
}T^j\omega_i
\]
which has $\triangle\ol\omega_i = (m/\lfloor m\rfloor)(T^{\lfloor
m\rfloor}-1)\omega
_i$; it is convenient to extend $\omega_i(x) = \omega_i(\lceil
n/r\rceil/m_n)$ for $x>\lceil n/r\rceil/m_n$. Decompose $\omega_i =
\ol\omega
_i+(\omega_i-\ol\omega_i)$. For the $\omega_1$-term,
\[
v^\dag\triangle\omega_1 v = \bigl(m/\lfloor m\rfloor\bigr)
v^\dag\bigl(T^{\lfloor
m\rfloor}\omega_1-\omega_1
\bigr) v + v^\dag\triangle(\omega_1-\ol\omega_1)v.
\]
By~(\ref{a23}) and Cauchy--Schwarz, the first term is bounded
absolutely by $(C_\varepsilon+\varepsilon\ol\eta)\llvert
v\rrvert^2$ and its integral by $\varepsilon\llVert v\rrVert
_*^2+C_\varepsilon\llVert v\rrVert^2$. The second term
calls for a summation by parts:
\begin{eqnarray*}
\bigl\langle v,\triangle(\omega_1-\ol\omega_1) v \bigr
\rangle&=& m_n \bigl( \bigl\langle v,(\omega_1-\ol
\omega_1)v \bigr\rangle- \bigl\langle Tv,(\omega_1-\ol
\omega_1)Tv \bigr\rangle\bigr)
\\
&=& m_n\Re\bigl\langle v-Tv,(\omega_1-\ol
\omega_1) (v+Tv) \bigr\rangle
\\
&=& \Re\bigl\langle Dv,(\ol
\omega_1-\omega_1) (v+Tv) \bigr\rangle.
\end{eqnarray*}
The averaged bound $\llvert\ol\omega_1 - \omega_1\rrvert
\le(C_\varepsilon+\varepsilon\ol\eta)^{1/2}$ and Cauchy--Schwarz
bound the integrand
\[
\bigl\llvert(Dv)^\dag(\ol\omega_1-\omega_1)
(v+Tv)\bigr\rrvert\le\sqrt{\varepsilon}\llvert Dv\rrvert^2+(1/4
\sqrt{\varepsilon}) (C_\varepsilon+\varepsilon\ol\eta) \bigl(\llvert
v\rrvert
^2+\llvert Tv\rrvert^2\bigr),
\]
and its integral by $\sqrt{\varepsilon}\llVert v\rrVert
_*^2+C_\varepsilon'\llVert v\rrVert^2$. One thus obtains a
similar bound on $\llvert \langle v,(\triangle\omega_1)
v \rangle\rrvert$.

There are corresponding bounds for the $\omega_2$-terms. For the $\ol
\omega_2$-term, use $2\llvert v\rrvert\llvert Tv\rrvert\le\llvert
v\rrvert^2+\llvert Tv\rrvert^2$.
For the $(\omega_2-\ol\omega_2)$-term, modify the summation by parts:
\begin{eqnarray*}
&& \Re\bigl\langle(v,\triangle(\omega_2-\ol\omega_2) Tv
\bigr\rangle
\\
&&\qquad = m_n\Re\bigl( \bigl\langle(v-Tv),(
\omega_2-\ol\omega_2)Tv \bigr\rangle+ \bigl\langle Tv,(
\omega_2-\ol\omega_2) \bigl(Tv-T^2 v\bigr)
\bigr\rangle\bigr)
\\
&&\qquad = \Re\bigl\langle Dv+TDv,(\ol\omega_2-\omega_2)Tv
\bigr\rangle.
\end{eqnarray*}
Incorporating all the $\omega$-terms into~(\ref{eta_u}), (\ref{eta_l})
and setting $\varepsilon$ small, we obtain~(\ref{dbound}) but with
the boundary terms omitted (from both the form and the norm).

We break the boundary term in~(\ref{Hnform}) into its positive and
negative parts. A~discrete analogue of a bound from Fact~\ref{f.1}
will be useful:
\[
\bigl\llvert v(0)\bigr\rrvert^2 = \int_0^{\infty}
-D \llvert v\rrvert^2 = \int_0^\infty
\Re m(v- T v)^\dag(v + T v) \le2\llVert D v\rrVert\llVert v\rrVert.
\]
It gives $\llvert v(0)\rrvert^2 \le(\varepsilon/b)\llVert D v\rrVert
^2+(b/\varepsilon)\llVert v\rrVert^2$,
and then $W^-\le b$ implies that
\[
0 \le v(0)^\dag W^- v(0) \le\varepsilon\llVert v\rrVert
_{*}^2+C''_\varepsilon
b^2\llVert v\rrVert^2
\]
which may be subtracted from the inequality already obtained. The
positive part may simply be added in using that $c\le1\le C$. We thus
arrive at~(\ref{dbound}).
\end{pf}

%
\begin{remark}If the $W_n$ are not bounded below then the lower bound
in~(\ref{dbound}) breaks down: in fact, the bottom eigenvalue of $H_n$
really goes to $-\infty$ like minus the square of the bottom
eigenvalue of $W_n$. This is the supercritical regime.
\end{remark}

\subsection{Convergence}

We begin with a simple lemma, a discrete-to-continuous version of Fact
\ref{f.2}.

%
\begin{lemma}\label{l.dc}
Let $f_n\in\F^n$ with $\llVert f_n\rrVert_{*n}$ uniformly
bounded. Then there exist $f\in L^*$ and a subsequence\vspace*{1pt} along which \textup{(i)}
$f_n\to f$ uniformly on compacts, \textup{(ii)}~$f_n\to_{L^2} f$, and \textup{(iii)}
$D_n f_n\to f'$ weakly in $L^2$.
\end{lemma}

\begin{pf}
Consider $g_n(x) = f_n(0) + \int_0^x D_n f_n$, a piecewise-linear
version of $f_n$; they coincide at points $x = i/m_n$, $i\in\Z_+$.
One easily checks that $\llVert g_n\rrVert^2_*\le2\llVert f_n\rrVert
^2_{*n}$, so some subsequence $g_n\to f\in L^*$
in all the modes of Fact~\ref{f.2}; for a Dirichlet component, the
boundary term in the $L^*_n$ norm guarantees that the limit vanishes at
0. But then also $f_n\to f$ compact-uniformly by a simple argument
using the uniform continuity of $f$, $f_n\to_{L^2} f$ because $\llVert
f_n-g_n\rrVert^2\le(1/3n^2)\llVert D_n f_n\rrVert
^2$, and $D_n f_n\to f'$ weakly in $L^2$ because $D_n f_n = g_n'$ a.e.
\end{pf}

Next, we establish a kind of weak convergence of the forms $
\langle\cdot,H_n\cdot\rangle$ to $\Hc(\cdot,\cdot)$. Let
$\Pc_n$ be orthogonal projection from $L^2$ onto $\F^n$ embedded as
above. The following facts will be useful and are easy to check. For
$f\in L^2$, $\Pc_n f \to_{L^2} f$ (the Lebesgue differentiation
theorem gives pointwise convergence and we have uniform
$L^2$-integrability); further, if $f'\in L^2$ then $D_n f\to_{L^2} f'$
($D_n f$ is a convolution of $f'$ with an approximate delta); for
smooth $\ph$, $\Pc_n \ph\to\ph$ uniformly on compacts. It is also
useful to note that $\Pc_n$ commutes with $R_n$ and with $D_n R_n$.
Finally, if $f_n\to_{L^2} f$, $g_n$ is $L^2$-bounded and $g_n\to g$
weakly in $L^2$, then $ \langle f_n,g_n \rangle\to
\langle f,g \rangle$.

%
\begin{lemma}\label{l.conv} Let $f_n\to f$ be as in the hypothesis and
conclusion of Lemma~\ref{l.dc}. Then for all $\ph\in C_0^\infty$ we
have $ \langle\ph,H_n f_n \rangle\to\Hc(\ph,f)$. In
particular, $\Pc_n\ph\to\ph$ in this way and so
%
%
\begin{equation}
\label{conv} \langle\Pc_n\ph,H_n\Pc_n \ph
\rangle= \langle\ph,H_n\Pc_n \ph\rangle\to\Hc(\ph,\ph).
\end{equation}
\end{lemma}

\begin{pf} Since $\ph$ is compactly supported, we have $R_n\ph=\ph
$ for $n$ large and the $R_n$s may be dropped. By assumption $D_n f_n$
is $L^2$ bounded and $D_n f_n\to f'$ weakly in $L^2$, so by the
preceding observations $D_n\ph\to_{L^2}\ph'$ and
\[
\bigl\langle\ph,D_n^\dag D_n f_n
\bigr\rangle= \langle D_n\ph,D_n f_n \rangle
\to\bigl\langle\ph',f' \bigr\rangle.
\]

For the potential term, we must verify that
\[
\langle\ph,V_n f_n \rangle= \bigl\langle\ph, \bigl(
\triangle_n Y_{n,1}+\tfrac{1}{2} \bigl((
\triangle_n Y_{n,2})T_n + T_n^\dag
\bigl(\triangle_n Y_{n,2}^\dag\bigr) \bigr)
\bigr)f_n \bigr\rangle
\]
converges to $- \langle\ph',Yf \rangle- \langle\ph
,Yf' \rangle$. Recall by Assumption~\ref{ass1} (\ref{a1}) and (\ref
{ar}) that $Y_{n,i}\to Y_i$ compact-uniformly ($i=1,2$) and $Y = Y_1
+\frac{1}{2}(Y_2+Y_2^\dag)$. Writing
$Y_n = Y_{n,1}+\frac{1}{2}(Y_{n,2}+Y^\dag_{n,2})\to Y$ (and\vspace*{1pt}
disregarding the
notational collision with $Y_i$), we first approximate $V_n$ by
$\triangle Y_n$:
\begin{eqnarray*}
\bigl\langle\ph,(\triangle_nY_n)f_n \bigr
\rangle&=& m_n \bigl( \langle\ph,Y_nf_n
\rangle- \langle T_n\ph,Y_nT_nf_n
\rangle\bigr)
\\
&=& m_n \bigl( \langle\ph,Y_nf_n \rangle-
\langle T_n\ph,Y_nf_n \rangle+ \langle
T_n\ph,Y_nf_n \rangle- \langle
T_n\ph,Y_nT_nf_n \rangle\bigr)
\\
&=& - \langle D_n\ph,Y_nf_n \rangle- \langle
T_n\ph,Y_nD_nf_n \rangle,
\end{eqnarray*}
which converges to the desired limit by the observations preceding the
lemma together with the assumptions on $f_n$ and the fact that $T_n\ph
\to_{L^2}\ph$ in $L^2$ since $m_n\llVert T_n\ph-\ph\rrVert= \llVert
D_n\ph\rrVert$ is bounded. The error in the
above approximation comes as a sum of $T_n$ and $T_n^\dag$ terms.
Consider twice the $T_n$ term:
\begin{eqnarray*}
\bigl\llvert\bigl\langle\ph,(\triangle_n Y_{n,2})
(T_n - 1)f_n \bigr\rangle\bigr\rrvert&=& \bigl\llvert
\bigl\langle\ph,\bigl(m_n^{-1}\triangle_n
Y_{n,2}\bigr)D_nf_n \bigr\rangle\bigr\rrvert
\\
&\le&\llVert\ph\rrVert\sup_I\bigl\llvert
Y_{n,2}-T_n^\dag Y_{n,2}\bigr\rrvert
\llVert D_n f_n\rrVert,
\end{eqnarray*}
where $I$ is a compact interval supporting $\ph$. (The single bars in
the supremum denote the spectral or $\ell_2$-operator norm, which is
of course equivalent to the max norm on the entries.) Note that $D_n
f_n$ is $L^2$-bounded because it converges weakly in $L^2$. Now
$Y_{n,2}$ and $T_n^\dag Y_{n,2}$ both converge to $Y_2$ uniformly on $I$,
in the latter case by the uniform continuity of $Y_2$ on $I$; it
follows that the supremum, and hence the whole term, vanish in the
limit. The $T_n^\dag$ term is handled similarly, the only difference
being that the $D_n$ in the estimate lands on $\ph$ instead.

Finally, for the boundary terms Assumption~\ref{ass3} gives
\[
(\Pc_n\ph)^*_i(0) w_{n,i} f_{n,i}(0)
\to\ph^*_i(0) w_i f_i(0),
\]
where in the Dirichlet case $i> r_0$ the left-hand side vanishes for $n$
large because $\ph_i$ is supported away from 0.

Turning to the second statement, we must verify that $\Pc_n\ph\to\ph
$ as in Lem\-ma~\ref{l.dc}. The uniform $L^*_n$ bound on $\Pc_n\ph$
follows from the following observations: $\llVert(\Pc_n \ph
){\sqrt{1+\ol\eta}} \rrVert=\llVert\Pc_n\ph{\sqrt
{1+\ol\eta}}\rrVert\le\llVert\ph
{\sqrt{1+\ol\eta}}\rrVert$; for $n$ large enough that $R_n\ph
=\ph$ we
have $\llVert D_n\Pc_n \ph\rrVert=\llVert\Pc_n
D_n\ph
\rrVert\le\llVert D_n\ph\rrVert\le\llVert\ph
'\rrVert$ (Young's inequality); for the boundary term note that
$(\Pc_n\ph)_i(0)$ is bounded if $i\le r_0$ and in fact vanishes for
$n$ large if $i>r_0$. The convergence is easy: $\Pc_n\ph\to\ph$
compact-uniformly and in $L^2$, and for $g\in L^2$ we have $
\langle g,D_n\Pc_n\ph\rangle= \langle\Pc_n g, D_n\ph
\rangle\to\langle g,\ph' \rangle$.
\end{pf}

We finish by recalling the argument to put all the pieces together. A
technical point: unlike in previous treatments we do not assume that
the eigenvalues are simple.

\begin{pf*}{Proof of Theorem~\ref{t.conv}}
We first show that for all $k$ we have $\underline\lambda_k=\liminf
\lambda_{n,k}\ge\Lambda_k$. Assume that $\underline\lambda
_k<\infty$. The eigenvalues of $H_n$ are uniformly bounded below by
Lemma~\ref{l.dbound}, so there is a subsequence along which $(\lambda
_{n,1},\ldots,\lambda_{n,k})\to(\xi_1,\ldots,\xi_k = \underline
\lambda_k)$. By the same lemma, corresponding orthonormal eigenvector
sequences have $L^*_n$-norm uniformly bounded. Pass to a further
subsequence so that they all converge as in Lemma~\ref{l.dc}. The
limit functions are orthonormal; by Lemma~\ref{l.conv} they are
eigenfunctions with eigenvalues $\xi_j\le\underline\lambda_k$ and
we are done.

We proceed by induction, assuming the conclusion of the theorem up to
$k-1$. For $j=0,\ldots,k-1$ let $v_{n,j}$ be orthonormal eigenvectors
corresponding to $\lambda_{n,j}$; for any subsequence we can pass to a
further subsequence such that $v_{n,j}\to_{L^2} f_j$, eigenfunctions
corresponding to $\Lambda_j$. Take an orthogonal eigenfunction $f_k$
corresponding to $\Lambda_k$ and find $f_k^\varepsilon\in C_0^\infty
$ with $\llVert f_k^\varepsilon- f_k\rrVert_*<\varepsilon$.
Consider the vector
\[
f_{n,k} = \Pc_n f_k^\varepsilon- \sum
_{j=0}^{k-1} \bigl\langle v_{n,j},
\Pc_n f_k^\varepsilon\bigr\rangle v_{n,j}.
\]
The $L^*_n$-norm of the sum term is uniformly bounded by $C\varepsilon
$: indeed, the $\llVert v_{n,j}\rrVert_{*n}$ are uniformly
bounded by Lemma~\ref{l.dbound}, while the coefficients satisfy
$\llvert \langle v_{n,j},f_k^\varepsilon\rangle\rrvert
\le\llVert f_k^\varepsilon-f_k\rrVert+\llVert
v_{n,j}-f_j\rrVert<2\varepsilon$ for large $n$. By the
variational characterization in finite dimensions and the uniform
$L^*_n$ form bound on $ \langle\cdot,H_n\cdot\rangle$
(by Lemma~\ref{l.dbound}) together with the uniform bound on $\llVert
\Pc_n f_k^\varepsilon\rrVert_{*n}$ (by Lemma~\ref
{l.conv}), we then have
%
%
\begin{eqnarray}\label{limsup}
\limsup\lambda_{n,k} &\le&\limsup\frac{ \langle f_{n,k},H_n
f_{n,k} \rangle}{ \langle f_{n,k},f_{n,k} \rangle}
\nonumber\\[-8pt]\\[-8pt]\nonumber
&=&
\limsup\frac{ \langle\Pc_n f_k^\varepsilon,H_n\Pc_n
f_k^\varepsilon\rangle}{ \langle\Pc_n f_k^\varepsilon
,\Pc_n f_k^\varepsilon\rangle}+o_\varepsilon(1),
\end{eqnarray}
where $o_\varepsilon(1)\to0$ as $\varepsilon\to0$. But~(\ref
{conv}) of Lemma~\ref{l.conv} provides $\lim\langle\Pc_n
f_k^\varepsilon,\break H_n\Pc_n f_k^\varepsilon\rangle= \Hc
(f_k^\varepsilon,f_k^\varepsilon)$,
so the right-hand side of~(\ref{limsup}) is
\[
\frac{\Hc(f_k^\varepsilon,f_k^\varepsilon)}{ \langle
f_k^\varepsilon,f_k^\varepsilon\rangle} + o_\varepsilon(1) = \frac{\Hc
(f_k,f_k)}{ \langle f_k,f_k \rangle}
+o_\varepsilon(1) = \Lambda_k + o_\varepsilon(1).
\]
Now letting $\varepsilon\to0$, we conclude $\limsup\lambda_{n,k}\le
\Lambda_k$.

Thus,\vspace*{1pt} $\lambda_{n,k}\to\Lambda_k$; Lemmas~\ref{l.dbound} and~\ref
{l.dc} imply that any subsequence of the $v_{n,k}$ has a further
subsequence converging in $L^2$ to some $f\in L^*$; Lemma~\ref{l.conv}
then implies that $f$ is an eigenfunction corresponding to $\Lambda
_k$. 
Finally, convergence is uniform over $W_n,W\ge-b$ since the bound~\ref
{l.dbound} is.
\end{pf*}

\section{CLT and tightness for Gaussian and Wishart models}\label{s.clt}

We now verify Assumptions \ref{ass1}--\ref{ass3} of Section~\ref{s.limits} for the band
Jacobi forms of Section~\ref{s.canonical}, and thus prove Theorems
\ref{t.G} and \ref{t.W} via Theorem~\ref{t.conv}.

We must consider the band forms as $(r\times r)$-block tridiagonal
matrices. This amounts to reindexing the entries by $(k+rj, l+rj)$,
where $j\in\Z_+$ indexes the blocks and $1\le k,l\le r$ give the
index within each block. The scalar processes obtained by fixing $k,l$
can then be analyzed jointly; finally, they can be assembled into a
matrix-valued process.

The technical tool we use to establish~(\ref{a1}) is a functional
central limit theorem for convergence of discrete time processes with
independent increments of given mean and variance (and controlled
fourth moments) to Brownian motion plus a nice drift. Appearing as
Corollary 6.1 in {RRV}, it is just a tailored version of a
much more general result given as Theorem~7.4.1 in \citet
{EthierKurtz}. We record it here.

%
\begin{proposition}\label{p.CLT} Let $a\in\R$ and $h\in C^1(\R_+)$,
and let $y_n$ be a sequence of discrete time\vspace*{1pt} real-valued processes with
$y_{n,0}=0$ and independent increments $\delta y_{n,j}=
y_{n,j}-y_{n,j-1}= m_n^{-1}\triangle_n y_{n,j}$.
Assume that $m_n\to\infty$ and
\begin{eqnarray*}
m_n\E\delta y_{n,j} &=& h'(j/m_n)+o(1),
\qquad m_n\E(\delta y_{n,j})^2 = a^2
+ o(1),
\\
m_n\E(\delta y_{n,j})^4 &=& o(1)
\end{eqnarray*}
uniformly for $j/m_n$ on compact sets as $n\to\infty$. Then
$y_n(x)=y_{n,\lfloor m_n x\rfloor}$ converges in law, with respect to
the compact-uniform topology, to the process $h(x)+ab_x$ where $b_x$ is
a standard Brownian motion.
\end{proposition}

%
\begin{remark} Since the limit is a.s. continuous, Skorokhod
convergence (the topology used in the references) implies uniform
convergence on compact intervals [see Theorem~3.10.2 in \citet
{EthierKurtz}] and we may as well speak in terms of the latter.
\end{remark}

\subsection{The Gaussian case}

Take $G_n = G_{n;0} + \sqrt{n}P_n$ as in~(\ref{Gp}) with $G_{n;0}$ as
in~(\ref{G}) and $P_n = \tilde P_n\oplus0_{n-r}$. We denote
upper-left $r\times r$ blocks with a tilde throughout. Set
\[
m_n = n^{1/3},\qquad H_n = \frac{m_n^2}{\sqrt{n}}
(2\sqrt{n} - G_n ).
\]
As usual, this soft-edge scaling can be predicted as follows. Centering
$G_n$ by $2\sqrt{n}$ gives, to first order, $\sqrt{n}$ times the
discrete Laplacian on blocks of size $r$. With space scaled down by
$m_n$, the Laplacian must be scaled up by $m_n^2$ to converge to the
second derivative. Finally, the scaling $m_n=n^{1/3}$ is determined by
convergence of the next order terms to the noise and drift parts of the
limiting potential.

Decompose $H_n$ as in~(\ref{Hn}), (\ref{Hn12}). The upper-left block is
\[
\tilde H_n = m_n^2 + m_n(W_n+Y_{n,1;0})
= m_n^2\bigl(2-n^{-1/2}\tilde G_{n,0}-
\tilde P_n\bigr);
\]
we want the boundary term $W_n$ to absorb the ``extra'' $m_n^2$ (the 2
in the right-hand side ``should be'' a 1) and the perturbation in order
to make $Y_{n,1;0}$ small just like the subsequent increments of
$Y_{n,i}$. We therefore set
\[
W_n = m_n (1-\tilde P_n ).
\]
With this choice Assumption~\ref{ass3} is an immediate consequence of the
hypotheses of Theorem~\ref{t.G}. The processes $Y_{n,1},Y_{n,2}$ are
determined and it remains to verify Assumptions \ref{ass1} and \ref{ass2}.

We begin with Assumption~\ref{ass1}, identifying the limiting integrated
potential $Y:\R_+\to M_r(\F)$ as that of the multivariate stochastic
Airy operator
%
%
\begin{equation}
\label{YSA} Y(x) = {\sqrt2}B_x + \tfrac{1}{2}rx^2,
\end{equation}
where $B_x$ is a standard $M_r(\F)$ Brownian motion and second term is
a scalar matrix.

\begin{pf*}{Proof of (\ref{a1}), Gaussian case}
Define scalar
processes $y_{k,l}$ for $1\le l\le r$ and $l\le k\le l+r$ by
%
%
\begin{equation}
\label{yY}y_{k,l} = %
\cases{ \displaystyle (Y_{n,1}
)_{k,l}, &\quad$l\le k\le r$,
\vspace*{3pt}\cr
\displaystyle\bigl(\tfrac{1}{2}Y_{n,2}^\dag
\bigr)_{k-r,l}, &\quad$r+1\le k\le l+r$.}
\end{equation}
(We have dropped the subscript $n$.) Equivalently, for $1\le k, l\le r$,
%
%
\begin{equation}
\label{Yy} (Y_{n,1} )_{k l} = %
\cases{
y^*_{ l,k}, &\quad$k\le l$,
\cr
y_{k, l},&\quad$k\ge l$,} \qquad
\bigl(\tfrac{1}{2}Y_{n,2}^\dag\bigr)_{k l}
= %
\cases{ y_{k+r, l}, &\quad$k\le l$,
\cr
0, &\quad$k> l$.}
\end{equation}
Then we have
%
%
\begin{equation}
\label{yG} \delta y_{k,l;j}=n^{-1/6} %
\cases{ \displaystyle -\frac{2}{\beta} \wt g_{k+rj}, &\quad$k= l$,
\vspace*{5pt}\cr
\displaystyle - g_{k+rj, l+rj}, &
\quad$l<k< l+r$,
\vspace*{5pt}\cr
\displaystyle \biggl(\sqrt{n}-\frac{1}{\sqrt\beta}\chi_{(n-k-rj+1)\beta}
\biggr), &\quad$k = l+r$.}
\end{equation}
Note that the $y_{k, l}$ are independent increment processes that are
mutually independent of one another. With the usual embedding
$j=\lfloor n^{1/3} x\rfloor$, Proposition \ref{p.CLT} together with
standard moment computations for Gaussian and Gamma random
variables---in particular
\begin{eqnarray*}
\E\chi_\alpha &=& \sqrt{\alpha}+O(1/\sqrt{\alpha}),\qquad\E(\chi
_\alpha-\sqrt{\alpha})^2 = 1/2 +O(1/\alpha),
\\
\E(\chi_\alpha-\sqrt{\alpha})^4 &=&O(1),
\end{eqnarray*}
for $\alpha$ large [valid since we consider $j = O(n^{1/3})$
here]---leads to the convergence of processes
\[
y_{k, l}(x) \Rightarrow%
\cases{ \displaystyle \sqrt{\frac{2}{\beta}} \wt
b_k(x), &\quad$k= l$,
\vspace*{3pt}\cr
b_{k, l}(x), &\quad$l<k< l+r$,
\vspace*{3pt}\cr
\displaystyle\frac{1}{\sqrt{2\beta}} b_k(x)+\frac{1}4 rx^2,&
\quad$k= l+r$,}
\]
where $b_k, \wt b_k$ are standard real Brownian motions and $b_{k,l}$
are standard $\F$ Brownian motions. By independence, the convergence
occurs jointly over $k,l$ and the limiting Brownian motions are all
independent. (For the $\F$ Brownian motions apply Proposition \ref
{p.CLT} to each of the $\beta$ real components, which are independent
of one another.) Therefore, $Y_{n,i}$ are both tight, and using~(\ref
{Yy}) we have, jointly for $1\le k, l\le r$,
\begin{eqnarray*}
&& \bigl(Y_{n,1}+\tfrac{1}{2}\bigl(Y_{n,2}^\dag+Y_{n,2}
\bigr) \bigr)_{k, l} =
\cases{ y_{k,k}+2y_{k+r,k},
\cr
y_{k,l}+y^*_{l+r,k},
\cr
y^*_{l,k}+y_{k+r,l}}
\\
&&\qquad \Rightarrow\quad
\cases{ \displaystyle\sqrt{\frac{2}{\beta}}(\wt b_k+
b_k)+\frac{1}{2}rx^2, &\quad$k= l$,
\vspace*{3pt}\cr
b_{k,l} +b^*_{l+r,k}, &\quad$k>l$,
\vspace*{3pt}\cr
b^*_{l,k}
+b_{k+r,l}, &\quad$k<l$.}
\end{eqnarray*}
Noting that the two Brownian motions in each entry are independent and
that the entries on and below the diagonal are independent of each
other, we conclude that this limiting matrix process is distributed as
$Y(x)$ in $(\ref{YSA})$.
\end{pf*}

We turn to Assumption~\ref{ass2}. Here, we need bounds over the full range $0\le
j\le\lceil n/r\rceil-1$. Recall that we can extend the $Y_{n,i}$
processes beyond the end of the matrix arbitrarily ($R_n$ takes care of
the truncation), and it is convenient to ``continue the pattern'' for
an extra block or two by setting $\chi_\alpha= 0$ for $\alpha<0$.
For the decomposition (\ref{Ydec}), we simply take $\eta_{n,i}$ to be
the expectation of $\triangle Y_{n,i}$ and $\triangle\omega_{n,i}$ to
be its centered version; the components of $\eta_{n,i}$ are then
easily estimated and those of $\omega_{n,i}$ become independent
increment martingales. We further set $\ol\eta(x)=rx$.

\begin{pf*}{Proof of (\ref{a21})--(\ref{a23}), Gaussian case}
From (\ref{yG}), we have $\eta_{n,1;j} = 0$ and
\[
(\eta_{n,2;j})_{k,l} = \E2m_n \delta
y_{k+r,l;j}=2n^{1/6} \bigl(\sqrt{n}-\beta^{-1/2}\E
\chi_{(n-k-r(j+1)+1)\beta} \bigr) 1_{k=l}. 
\]
The estimate
%
%
\begin{equation}
\label{chi} \sqrt{(\alpha-1)^+} \le\E\chi_\alpha= {\sqrt2}
\frac{\Gamma
((\alpha+1)/2)}{\Gamma(\alpha/2)} \le\sqrt{\alpha}
\end{equation}
is useful. We obtain
\[
2n^{1/6}\frac{rj-c}{2\sqrt{n}}\le(\eta_{n,2;j})_{k,k} \le
2n^{1/6}\frac{rj+c}{\sqrt{n}}
\]
for some fixed $c$, which yields the matrix inequalities
\[
rx - cn^{-1/3}\le\eta_{n,2}(x)\le2rx+cn^{-1/3}
\]
and verifies (\ref{a21}) with $\ol\eta(x)=rx$. Separately, we have
the upper bound (\ref{a22}):
\[
\eta_{n,2}(x)\le2n^{2/3} = 2m_n^2.
\]

The bound (\ref{a23}) may be done entry by entry, so we consider the
process $\{(\omega_{i,n;j})_{k,l}\}_{j\in\Z_+}$ for fixed $i=1,2$
and $1\le k,l\le r$ and further omit these indices; for the $\F
$-valued processes we restrict attention further to one of the $\beta$
real-valued components, and denote the latter simply by $\omega
_{n;j}$. Consider (\ref{yG}); the key points are that the increments
$\delta\omega_{n;j}$ are independent and centered, and that scaled up
by $n^{1/6} = m_n^{1/2}$ they have uniformly bounded fourth moments. To
prove (\ref{a23}), it is enough to consider $x$ at integer points and
show that the random variables
\[
\sup_{x = 0,1,\ldots,n/rm_n} x^{\varepsilon-1} \sup_{j=1,\ldots,
m_n}
\llvert\omega_{n;m_nx+j}-\omega_{n;m_nx}\rrvert^2
\]
are tight over $n$. Squaring, bounding the outer supremum by the
corresponding sum, and then taking expectations gives
\[
\sum_{x=0}^{n/rm_n} \frac{\E\sup_{j=1,\ldots, m_n}\llvert
\omega_{n;m_nx+j}-\omega_{n;m_nx}\rrvert^4}{x^{2-2\varepsilon
}}\le\sum
_{x=0}^{n/rm_n} \frac{16\E\llvert\omega
_{n;m_n(x+1)}-\omega_{n;m_nx}\rrvert^4}{x^{2-2\varepsilon}},
\]
where we have used the $L^p$ maximum inequality for martingales
[see, e.g., Proposition 2.2.16 of \citet{EthierKurtz}]. To bound
the latter
expectation, expand the fourth power to obtain $O(m_n^2)$ nonzero
terms that are $O(m_n^{-2})$ with constants independent of $x$ and $n$.
It follows that the entire sum is uniformly bounded over~$n$, as required.
\end{pf*}

\subsection{The Wishart case}

Take $L_{n,p} = \Sigma^{1/2}_{n,p}L_{n,p,0}$ with $L_{n,p,0}$ as
in~(\ref{L}) and, denoting the upper-left $r\times r$ block with a
tilde, $\Sigma_{n,p}=\tilde\Sigma_{n,p}\oplus I_{n\wedge p}$. Recall
that $L_{n,p}$ is $((n+r)\wedge p)\times(n\wedge p)$. Put $S_{n,p} =
L_{n,p}^\dag L_{n,p}$ and similarly for $S_{n,p,0}$; these matrices are
$(n\wedge p)\times(n\wedge p)$ and the latter is given explicitly in
(\ref{S}). We sometimes drop the subscripts $n,p$. Recall (\ref{Sp})
that $S-S_0 = \tilde L_0^\dag(\tilde\Sigma-1)\tilde L_0\oplus0$.

We set
%
%
\begin{equation}
\label{mHW}
\qquad m_{n,p} = \biggl(\frac{\sqrt{np}}{\sqrt{n}+\sqrt{p}} \biggr)^{2/3},
\qquad H_{n,p} = \frac{m_{n,p}^2}{\sqrt{np}} \bigl( (\sqrt{n}+\sqrt{p}
)^2 - S_{n,p} \bigr).
\end{equation}
See Part~I for detailed heuristics behind the scaling;
written in this way, it allows that $p,n\to\infty$ together
arbitrarily, that is, only $n\wedge p\to\infty$. It is useful to note that
\[
2^{-2/3}(n\wedge p)^{1/3} \le m_{n,p} \le(n\wedge
p)^{1/3}.
\]

Decompose $H_{n,p}$ as in (\ref{Hn}), (\ref{Hn12}). The upper-left
block is
\[
\tilde H = m^2 + m(W+Y_{1;0}) = 2m^2-
\frac{m^2}{\sqrt{np}} \bigl(\tilde S_{0}-n-p+\tilde L_0^\dag(
\tilde\Sigma-1)\tilde L_0 \bigr).
\]
As before we want $W$ to absorb the extra $m^2$ and the perturbation in
order to make $Y_{1;0}$ small. Now the perturbation term is random, but
it does not have to be fully absorbed; it is enough that $Y_{1;0}\to0$
in probability. The reason is that the process $Y_1$ can absorb an
overall additive random constant that tends to zero in probability, as
is clear in Assumption~\ref{ass1} while in Assumption~\ref{ass2} the constant may be put
into $\omega_1$. Since $\tilde L_0\approx\sqrt{n}$, we set
%
%
\begin{equation}
\label{WW} W_{n,p} = m_{n,p} \bigl(1-\sqrt{n/p} (\tilde
\Sigma_{n,p}-1 ) \bigr).
\end{equation}
Once again, Assumption~\ref{ass3} follows immediately from the hypotheses of
Theorem~\ref{t.W}.

We must still deal with the perturbed term in $Y_{1;0}$ and show that
%
%
\begin{equation}
\label{y10} \frac{m}{\sqrt{np}} \bigl(n\tilde\Sigma-\tilde L_0^\dag
\tilde\Sigma\tilde L_0 \bigr)\to0
\end{equation}
in probability. We defer this to the end of the proof of Assumption~\ref{ass1},
to which we now turn. As in the Gaussian case, $Y$ is given by (\ref{YSA}).

\begin{pf*}{Proof of (\ref{a1}), Wishart case}
By the preceding
paragraph it suffices to treat the null case $\Sigma= I$ and
afterward check (\ref{y10}). Define processes $y_{k,l}$ for $1\le
l\le r$ and $l\le k\le l+r$ by (\ref{yY}) as in the Gaussian case.
From (\ref{S}) with the centering and scaling of (\ref{mHW}) and
(\ref{Hn12}), we obtain
\[
\delta y_{k,l;j} = \frac{m}{\sqrt{np}} %
\cases{
\displaystyle n+p- \frac{1}{\beta} \bigl(\wt\chi_{(n-k-rj+1)\beta}^2 + \chi_{(p-k-r(j+1)+1)\beta}^2 \bigr) +O(1),
\cr
\hspace*{189pt} k=l,
\cr
\displaystyle -\frac{1}{\sqrt\beta} \bigl(\wt\chi_{(n-k-rj+1)\beta }g_{k+rj,l+rj}
\cr
\displaystyle \qquad{}+ \chi_{(p-l-r(j+1)+1)\beta}g^*_{l+r(j+1),k+rj} \bigr)+O(1),
\cr
\hspace*{189pt} l<k<l+r,
\cr
\displaystyle \sqrt{np}-\frac{1}\beta\wt\chi_{(n-k-rj+1)\beta}\chi_{(p-k-rj+1)\beta},\qquad k=l+r,}
\]
where the $O(1)$ terms stand in for the interior Gaussian sums of
(\ref{S}), all of whose \emph{moments} are bounded uniformly in
$n,p$. Since $m^{1+k}/(np)^{k/2} \le m^{1-2k} = o(1)$ for $k\ge1$,
these terms are negligible in the scaling of Proposition~\ref{p.CLT}
in the sense that the associated processes converge to the zero
process. Next, use that expressions of type $\chi_n-\sqrt{n}$ are
$O(1)$ in the same sense, and that $\sqrt{n}-\sqrt{n-j}=O(j/\sqrt
{n})=O(m/\sqrt{n}) = o(1)$ since we consider $j/m$ bounded here (and
similarly for $p$), to write\vspace*{-3pt}
%
%
\begin{equation}
\label{dyW}
\delta y_{k,l;j} = \frac{m}{\sqrt{np}} %
\cases{
\displaystyle \frac{2}{\sqrt\beta} \bigl(\sqrt{n}(\sqrt{\beta n}-\wt\chi_{(n-k-rj+1)\beta})
\cr
\displaystyle \qquad{} +\sqrt{p}(\sqrt{\beta p}-\chi
_{(p-k-r(j+1)+1)\beta}) \bigr) +O(1),
\vspace*{2pt}\cr
\qquad k=l,
\vspace*{2pt}\cr
\displaystyle -\sqrt{n} g_{k+rj,l+rj} -
\sqrt{p} g^*_{l+r(j+1),k+rj}+O(1),
\vspace*{2pt}\cr
\qquad l<k<l+r,
\cr
\displaystyle \frac{1}{\sqrt\beta} \bigl(
\sqrt{p}(\sqrt{\beta n}-\wt\chi_{(n-k-rj+1)\beta})
\vspace*{2pt}\cr
\displaystyle\qquad{} +\sqrt{n}(\sqrt{
\beta p}-\chi_{(p-k-rj+1)\beta}) \bigr)+O(1),
\cr \qquad k=l+r.}
\end{equation}

It suffices to prove tightness and convergence in distribution along a
subsequence of any given subsequence, and we may therefore assume that
$p/n\to\gamma^2\in[0,\infty]$. Each case of (\ref{dyW}) contains
two terms, and each one of these terms forms an independent increment
process to which Proposition \ref{p.CLT} may be applied. (Break the
$\F$-valued terms up further into their real-valued parts.) Standard
moment computations as in the Gaussian case, together with
independence, then lead to the joint convergence of processes
\[
y_{k, l}(x) \Rightarrow%
\cases{
\displaystyle \sqrt{\frac{2}{\beta}}
\biggl(\frac{1}{1+\gamma}\wt b_k(x)+\frac{\gamma}{1+\gamma}b_k(x)
\biggr)+\frac{\gamma
}{(1+\gamma)^2}rx^2,
\cr
\hspace*{166pt} k= l,
\cr
\displaystyle \frac{1}{1+\gamma}b_{k, l}(x)+\frac{\gamma}{1+\gamma
}b^*_{l+r,k}(x), \qquad l<k< l+r,
\cr
\displaystyle\frac{1}{\sqrt{2\beta}} \biggl(\frac{\gamma}{1+\gamma}\wt
b_k(x)+\frac{1}{1+\gamma}b_k(x) \biggr)+
\frac{1+\gamma^2}{4(1+\gamma
)^2} rx^2,
\cr
\hspace*{166pt} k= l+r,}
\]
where $b_k, \wt b_k$ are standard real Brownian motions and $b_{k,l}$
are standard $\F$ Brownian motions, all independent except that
$b_{k+r,l+r}$ and $b_{k,l}$ are identified. Therefore, $Y_{n,i}$ are
both tight. Furthermore, using~(\ref{Yy}) we have
\begin{eqnarray*}
&& \bigl(Y_{n,1}+\tfrac{1}{2}\bigl(Y_{n,2}^\dag+Y_{n,2}
\bigr) \bigr)_{k, l} =
\cases{ y_{k,k}+2y_{k+r,k},
\cr
y_{k,l}+y^*_{l+r,k},
\cr
y^*_{l,k}+y_{k+r,l}}
\\
&&\qquad \Rightarrow\quad
\cases{ \displaystyle\sqrt{\frac{2}{\beta}}(\wt b_k+
b_k)+\frac{1}{2}rx^2, &\quad$k= l$,
\vspace*{3pt}\cr
b_{k,l} +b^*_{l+r,k}, &\quad$k>l$,
\vspace*{3pt}\cr
b^*_{l,k}
+b_{k+r,l}, &\quad$k<l$}
\end{eqnarray*}
jointly for $1\le k, l\le r$. After the dust clears, we thus arrive at
exactly the same limiting process as in the Gaussian case, namely
(\ref{YSA}).

We now address (\ref{y10}). Here, we can replace $\tilde L_0$ with
$\sqrt{n}I_r$ at the cost of an error that has uniformly bounded
second and fourth moments. Now (\ref{WW}) and the assumed lower bound
on $W_{n,p}$ give that $\tilde\Sigma\le1+2\sqrt{p/n}$ for $n,p$
large; this matrix inequality holds entrywise in the diagonal basis for
$\tilde\Sigma$ (which was fixed over $n,p$). One therefore obtains
error terms with mean square $O(m^2/n+m^2/p) = O(m^{-1})$ which is
$o(1)$ as required.
\end{pf*}

Turning to Assumption~\ref{ass2}, we may continue the processes $Y_{n,i}$ past
the end of the matrix for convenience just as in the Gaussian case. The
Wishart case presents an additional issue at the ``end'' of the matrix:
recall that the final $r$ rows and columns of $S$ in (\ref{S}) may
have some apparently nonzero terms set to zero. However, these changes
are easily absorbed into the bounds that follow. For (\ref{Ydec}), we
once again take $\eta_{n,i}$ to be the expectation of $\triangle
Y_{n,i}$ and $\triangle\omega_{n,i}$ to be its centered version. We
also set $\ol\eta(x)=rx$ as before.

\begin{pf*}{Proof of (\ref{a21})--(\ref{a23}), Wishart case}
This time we have
\begin{eqnarray*}
(\eta_{n,1;j})_{k,l} &=& \E m \delta y_{k,l;j}=m^2(np)^{-1/2}
(2rj-r+1 ) 1_{k=l},
\\
(\eta_{n,2;j})_{k,l} &=& \E2m \delta y_{k+r,l;j}
\\
&=&2m^2
\bigl(1-\beta^{-1}(np)^{-1/2}\E\wt\chi_{(n-k-rj+1)\beta}
\chi_{(p-k-rj+1)\beta
} \bigr)1_{k=l}.
\end{eqnarray*}
Using (\ref{chi}) one finds, for some constant $c$, that
\[
m^{-1}(rj+c)\le(\eta_{n,1;j}+\eta_{n,2;j})_{k,k}
\le m^{-1}(2rj+c)
\]
which yields (\ref{a21}) with $\ol\eta(x)=rx$. Separately, we have
the upper bound (\ref{a22}). The oscillation bound (\ref{a23}) may be
proved exactly as in the Gaussian case: we have once again that $\{
\sqrt{m_n}(\omega_{n,i;j})_{k,l}\}_{j\in\Z_+}$ are martingales with
independent increments whose fourth moments are uniformly bounded.
\end{pf*}

\section{Alternative characterizations of the laws}\label{s.alternative}

In this section, we derive the SDE and PDE characterizations, proving
Theorems \ref{t.SDE} and \ref{t.PDE}.

\subsection{First-order linear ODE}

For each noise path $B_x$, the eigenvalue equation $\Hc_{\beta,W} f =
\lambda f$ can be rewritten as a first-order linear ODE with continuous
coefficients. We begin with the formal second-order linear differential equation
%
%
\begin{equation}
\label{SL} f''(x) = \bigl(x-\lambda+{
\sqrt2}B'_x\bigr)f(x),
\end{equation}
where $f:\R_+\to\F^r$, with initial condition
%
%
\begin{equation}
\label{IC} f'(0) = W f(0).
\end{equation}
As usual, we allow $W\in M_r^*(\F)$ and interpret (\ref{IC}) via
(\ref{Wrig}). Rewrite (\ref{SL}) in the form
\[
\bigl(f' - {\sqrt2}Bf \bigr)' = (x-\lambda)f - {
\sqrt2}Bf'.
\]
Now let $g = f'-{\sqrt2}Bf$. The equation becomes
\begin{eqnarray*}
g' &=& (x-\lambda)f-{\sqrt2}Bf'
\\
&=& \bigl(x-\lambda-2B^2\bigr)f-{\sqrt2}Bg.
\end{eqnarray*}
In other words, the pair $ (f(x),g(x) )$ formally satisfies
the first-order linear system
%
%
\begin{equation}
\label{sys} \lleft[\matrix{ f'
\cr
g'} \rright] = \lleft[\matrix{ {\sqrt2}B&1
\cr
x-\lambda-2B^2&-{
\sqrt2}B} \rright] \lleft[\matrix{ f
\cr
g} \rright].
\end{equation}
Since $B_0=0$, $g$ simply replaces $f'$ in the initial condition (\ref
{IC}). If one prefers, this condition can be written in the standard form
%
%
\begin{equation}
\label{sysIC} -\tilde Wf(0) + \tilde Ig(0) = 0,
\end{equation}
where $\tilde W = \sum_{i\le r_0}w_iu_iu_i^\dag+\sum_{i>r_0} u_i
u_i^\dag$ and $\tilde I = \sum_{i\le r_0}u_i u_i^\dag$.

One could allow general measurable coefficients and define a solution
to be a pair of absolutely continuous functions $(f,g)$
satisfying~(\ref{sys}) Lebesgue a.e. This definition, equivalent to
writing (\ref{sys}) in an integrated form, is easily seen to coincide
with (\ref{integrated}). As in Remark \ref{r.ee}, however, we note
the coefficients are continuous; solutions may therefore be taken to
satisfy~(\ref{sys}) everywhere and are in fact continuously
differentiable. It is classical that the initial value problem has a
unique solution which exists for all $x\in\R_+$ (and further depends
continuously on the parameter $\lambda$ and the initial condition $W$).

\subsection{Matrix oscillation theory}

The matrix generalization of Sturm oscillation theory goes back to the
classic work of Morse \citet{Morse} [see also \citet{Morse2}].
Textbook treatments of self-adjoint differential systems include that
of \citet{Reid}. Our reference will be the paper of \citet
{BK}, which allows sufficiently general boundary conditions.

We first consider the eigenvalue problem on a finite interval $[0,L]$
with Dirichlet boundary condition $f(L)=0$ at the right endpoint. In
the scalar-valued setting, the number of eigenvalues below $\lambda$
is found to coincide with the number of zeros of $f$ (the solution of
the initial value problem) that lie in $(0,L)$. The correct
generalization to the matrix-valued setting involves tracking a matrix
whose columns form a basis of solutions, and counting the so-called
``focal points''.

We need a little terminology and a few facts from \citet{BK},
especially Definition 1 on page 338 there and the points that follow.
A \textit{matrix solution} of (\ref{sys}) is a pair $F,G:\R_+\to\F
^{r\times r}$ such that each column of $ \bigl[
{\fontsize{8.36pt}{9pt}\selectfont{\matrix{F\cr G}}}
\bigr]$ is a solution. A~\textit{conjoined basis} for (\ref{sys})
is a matrix solution $(F,G)$ with the additional properties that
$F^\dag
G=G^\dag F$ and $\operatorname{rank} \bigl[
{\fontsize{8.36pt}{9pt}\selectfont{\matrix{
F\cr G}}}
\bigr] = r$. The latter properties hold identically on $\R_+$ as
soon as they do at a single point; in particular, we may set $F(0)
=\tilde I$ and $G(0) = \tilde W$ to obtain a conjoined basis for the
initial condition (\ref{sysIC}). A point $x\in\R_+$ is called a
\textit{focal point} if $F(x)$ is singular, of \textit{multiplicity}
$\operatorname{nullity} F(x)$. The following proposition summarizes what
we need from the more general results of \citet{BK}.

%
\begin{proposition} Consider the differential system
\[
\lleft[\matrix{ f'
\cr
g'} \rright] =
\lleft[\matrix{ A&B
\cr
C-C_0\lambda&-A^\dag} \rright] \lleft[\matrix{ f
\cr
g } \rright]
\]
with real parameter $\lambda$, where $A(x),B(x),C(x),C_0(x)$ are
$n\times n$ matrices depending continuously on $x\in\R$ with
$B,C,C_0$ Hermitian and $B,C_0 > 0$. For each $\lambda\in\R$, let
$(F,G)$ be a conjoined basis with some fixed initial condition at 0.
Consider also the associated eigenvalue problem on $[0,L]$ with the
same boundary condition at 0 and Dirichlet condition $f=0$ at $L$.
Then, for all $\lambda\in\R$, the number of focal points of $(F,G)$
in $(0,L)$ equals the number of eigenvalues below $\lambda$.
Furthermore, the spectrum is purely discrete and bounded below with
eigenvalues tending to infinity.
\end{proposition}

\begin{pf} The idea is that focal points are isolated and move
continuously to the left as $\lambda$ increases. For sufficiently
negative $\lambda$, there are no focal points on $(0,L]$; each time
$\lambda$ passes an eigenvalue, a new focal point is introduced at $L$.

We indicate how the proposition follows from the results of \citet
{BK}. Note that Conditions (A1), (A2) on page~337 are satisfied by our
coefficients, and that (A3) on page~340 is satisfied by our boundary
conditions. Theorem~1 on page 345 thus applies. See (3.5) on page~341
for the definition of $\Lambda(\lambda)$; the Dirichlet condition at
$L$ gives the particularly simple result that the right-hand side of
(4.1) vanishes, so the quantity $n_2(\lambda)$ is constant. Theorem~2
applies as well, and we obtain $n_1(\lambda) -n_1 = n_3(\lambda)$.
Here, $n_1(\lambda)$ is the number of focal points in $[0,L)$, $n_1 =
\lim_{\lambda\to-\infty}n_1(\lambda)$ and $n_3(\lambda)$ is the
number of eigenvalues below $\lambda$. To finish, we consult Theorem~3
on page~353; noting that (A4$'$) is satisfied by Section~7.2, page~365,
to find that $n_1$ is simply the multiplicity of the focal point at~0.
The oscillation result follows. For the assertion about the spectrum,
we apply Theorem~4, noting that (A5), page~358 holds by (i) there, and
(A6), page~359 also holds.
\end{pf}

We conclude the following for our matrix system.

%
\begin{lemma}\label{l.osc} Consider the eigenvalue problem (\ref
{sys}) on $[0,L]$ with boundary conditions (\ref{sysIC}) and $f(L) =
0$. For each $\lambda\in\R$, let $(F,G)$ be the conjoined basis
initialized by $F(0) =\tilde I$ and $G(0) = \tilde W$; then the number
of focal points in the interval $(0,L)$, counting multiplicity, equals
the number of eigenvalues below $\lambda$. Furthermore, the spectrum
is purely discrete and bounded below with eigenvalues tending to infinity.
\end{lemma}

A soft argument now recovers an oscillation theorem for the original
half-line problem.

%
\begin{theorem}\label{t.osc}Consider the eigenvalue problem (\ref
{sys}), (\ref{sysIC}) on $L^2(\R_+)$. For each $\lambda\in\R$, let
$(F,G)$ be the conjoined basis as above; then the number of focal
points in $(0,\infty)$ equals the number of eigenvalues strictly below
$\lambda$.
\end{theorem}

\begin{pf}
Let $\Lambda_{L,k}, \Lambda_k$, $k=0,1,\ldots$ denote the lowest
eigenvalues of the truncated and half-line operators $\Hc_L, \Hc$,
respectively; it suffices to show that $\lim_{L\to\infty}\Lambda
_{L,k}=\Lambda_k$ for each $k$. Indeed, taking $L\to\infty$ in Lemma
\ref{l.osc} then yields the conclusion for each $\lambda\in\R
\setminus\{\Lambda_0,\Lambda_1,\ldots\}$. Letting $\lambda\searrow
\Lambda_k$, the right-most focal point must tend to $\infty$ by
monotonicity and continuity, so the claim actually holds for all
$\lambda\in\R$.

The variational problem for $\Hc_L$ simply minimizes over the subset
of $L^*$ functions that vanish on $[L,\infty)$; the Dirichlet
condition is important here. It follows immediately that $\Lambda
_{L,k}\ge\Lambda_k$, using the min--max formulation of the variational
characterization. Proceed by induction, assuming that $\Lambda
_{L,j}\to\Lambda_L$ for $j=0,\ldots, k-1$.

Let $f_{L,j}$ be orthonormal eigenvectors corresponding to $\Lambda
_{L,j}$. By the induction hypothesis, the variational characterization
for $\Hc$ and the finite-dimensionality of its eigenspaces, every
subsequence has a further subsequence such that $f_{L,j}\to_{L^2}
f_j$, eigenvectors corresponding to $\Lambda_j$. Let $f_k$ be an
orthogonal eigenvector corresponding to $\Lambda_k$ and take
$f_k^\varepsilon$ compactly supported with $\llVert
f_k^\varepsilon-f_k\rrVert_*<\varepsilon$. Let
\[
g_L = f_k^\varepsilon-\sum
_{j=0}^{k-1} \bigl\langle f_k^\varepsilon
,f_{L,j} \bigr\rangle f_{L,j}.
\]
For large $L$, the inner products are at most $2\varepsilon$, so
$\llVert g_L-f_k\rrVert_*\le c\varepsilon$. Noting that
$g_L$ is eventually supported on $[0,L]$, the variational
characterization gives
\[
\limsup_{L\to\infty} \Lambda_{L,k}\le\limsup
_{L\to\infty}\frac
{\Hc(g_L,g_L)}{ \langle g_L,g_L \rangle}
\]
and the right-hand side tends to ${\Hc(f_k,f_k)}/{ \langle
f_k,f_k \rangle} = \Lambda_k$ as $\varepsilon\to0$.
\end{pf}

\subsection{Riccati SDE: Stochastic airy meets dyson}

Let $(F,G)$ be a conjoined basis for (\ref{sys}) as defined in the
previous subsection. Then, on any interval with no focal points, the
matrix $Q = GF^{-1}$ is self-adjoint and satisfies the \textit{matrix
Riccati equation}
%
%
\begin{equation}
\label{Q} Q' = rx-\lambda-(Q+{\sqrt2}B)^2
\end{equation}
[see page~338 of \citet{BK}].

As $x$ passes through a focal point $x_0$, an eigenvalue $q$ of $Q$
``explodes to $-\infty$ and restarts at $+\infty$''. The precise
evolution of $Q$ near $x_0$ can be seen by choosing $a\in\R$ so that
$\tilde Q = (Q-a)^{-1}= F(G-aF)^{-1}$ is defined; then $\tilde Q$ satisfies
%
%
\begin{equation}
\label{Qt} \tilde Q' = \bigl(1+\tilde Q({\sqrt2}B+a) \bigr)
\bigl(1+({\sqrt2}B+a)\tilde Q \bigr) - (x-\lambda)\tilde Q^2.
\end{equation}
Writing $\tilde q =1/(q-a)$ and $v$ for the corresponding eigenvector,
notice how
\[
\tilde q'(x_0) = v(x_0)^\dag
\tilde Q'(x_0) v(x_0) = 1.
\]
Thus, $\tilde q$ is ``pushed up through zero'', corresponding to the
explosion/restart in $q = 1/\tilde q + a$. In this way, we may consider
$Q(x)\in M_r^*(\F)$ to be defined for all $x$. The initial condition
is then simply $Q(0) = W$.

Now let $P = F'F^{-1}$. While $P = Q + {\sqrt2}B$ is not
differentiable, by
(\ref{Q}) it certainly satisfies the integral equation
\[
P_{x_2} - P_{x_1} = {\sqrt2}(B_{x_2}-B_{x_1})+
\int_{x_1}^{x_2} \bigl(ry - \lambda-
P_y^2\bigr) \,dy
\]
if $[x_1,x_2]$ is free of focal points. In other words, $P$ \emph{is a
strong solution of the It\^o equation}
%
%
\begin{equation}
\label{matSDE} dP_x ={\sqrt2} dB_x+ \bigl(rx-
\lambda-P_x^2\bigr)\,dx
\end{equation}
off the focal points. The evolution of $P$ through a focal point can be
described in the coordinate $\tilde P = (P-a)^{-1} = F(F'-aF)^{-1}$.
Using (\ref{Qt}) and It\^o's lemma, one could write down an SDE for
$\tilde P = \tilde Q(1+{\sqrt2}B\tilde Q)^{-1}$. The initial condition here
is also $P(0)=W$.

Consider the eigenvalues $p_1, \ldots, p_r$ of $P$. The main point is
that the drift term in (\ref{matSDE}) is unitarily equivariant and
passes through the usual derivation of Dyson's Brownian motion
[\citet{Dyson}]. The eigenvalues therefore evolve as an
autonomous Markov process.

To describe the law on paths we need a space, and there are two issues:
it will be necessary to keep the eigenvalues ordered but also allow for
explosions/restarts. We therefore define a sequence of \textit{Weyl
chambers} $C_k\subset(-\infty,\infty]^r$ by
\begin{eqnarray*}
C_0 &=& \{p_1 < \cdots< p_r\},
\\
C_1 &=& \{p_2 <\cdots< p_r< p_1\},
\\
C_2 &=& \{p_3< \cdots< p_r< p_1<p_2\}
\end{eqnarray*}
and so on, permuting cyclically. We glue successive adjacent chambers
together at infinity in the natural way to make the disjoint union $\Cc
= C_0\cup C_1\cup\ldots$ into a connected smooth manifold. That is,
taking $p_1\to-\infty$ in $C_0$ puts you at $p_1=+\infty$ in $C_1$;
the smooth structure is defined by the coordinate $\tilde p_1=1/p_1$,
which vanishes along the seam. Glue $C_{k-1}$ to $C_{k}$ similarly
along $\{p_{ k~\mathrm{mod}~r}=\infty\}$. We also define $\ol C_k$,
$\ol\Cc$ in which some coordinates may be equal, and $\del C_k = \ol
C_k\setminus C_k$, $\del\Cc=\ol\Cc\setminus\Cc$ in which some
coordinates are equal.

%
\begin{theorem}\label{t.Dyson} Represent the eigenvalues of\/ $W\in
M_r^*(\F)$ as\/ $\ww=\break (w_1\ldots, w_r)\in\ol C_0$. The eigenvalues
$\pp=(p_1,\ldots,p_r)$ of $P$ evolve as an autonomous Markov process
whose law on paths $\R_+\to\ol\Cc$ is the unique weak solution of
the SDE system
%
%
\begin{equation}
\label{evSDE} dp_i = \frac{2}{\sqrt{\beta}}\,db_i +
\biggl(rx-\lambda-p_i^2+\sum
_{j\neq i}\frac{2}{p_i-p_j} \biggr)\,dx
\end{equation}
with initial condition $\pp(0) = \ww$, where $b_1,\ldots, b_r$ are
independent standard real Brownian motions. An eigenvalue $p_i$ can
explode to $-\infty$ and restart at $+\infty$, meaning $\pp$ crosses
from $C_{k}$ to $C_{k+1}$; the evolution through an explosion is
described in the coordinate $\tilde p_i = 1/p_i$, which satisfies
%
%
\begin{equation}
\label{evSDE2} \qquad d\tilde p_i = -\frac{2}{\sqrt{\beta}}\tilde
p_i^2\,db_i + \biggl(1+ \biggl(\lambda-rx+
\sum_{j\neq i}\frac{2\tilde p_i\tilde
p_j}{\tilde p_i-\tilde p_j} \biggr)\tilde
p_i^2+\frac{4}{\beta}\tilde p_i^3
\biggr)\,dx.
\end{equation}
\end{theorem}

\begin{pf}
Deriving (\ref{evSDE}) from (\ref{matSDE}) is simply a matter of
applying It\^o's lemma, at least in $\Cc$ where the eigenvalues are
distinct. One needs to differentiate an eigenvalue with respect to a
matrix, and this information is given by Hadamard's variation formulas.
In detail, let $A\in M_r(\F)$ vary smoothly in time and suppose $A(0)$
has distinct spectrum. Then eigenvalues $\lambda_1,\ldots,\lambda_r$
of $A$ and corresponding eigenvectors $v_1,\ldots,v_r$ vary smoothly
near 0 by the implicit function theorem. Differentiating $Av_i =
\lambda_i v_i$ and $v_i^\dag v_i=1$ lead to the formulas
\[
\dot\lambda_i = v_i^\dag\dot
Av_i,\qquad\qquad\ddot\lambda_i = v_i^\dag
\ddot Av_i+2\sum_{j\neq i}
\frac{\llvert v_i^\dag\dot
Av_j\rrvert^2}{\lambda_i-\lambda_j}.
\]
Writing $X = \dot A(0)$ and $\nabla_X$ for the directional derivative,
and taking\break $v_1(0),\ldots, v_r(0)$ to be the standard basis, we find
\[
\nabla_X\lambda_i = X_{ii},\qquad\qquad
\nabla_X^2\lambda_i = 2\sum
_{j\neq i}\frac{\llvert X_{ij}\rrvert^2}{\lambda
_i-\lambda_j}.
\]
Returning to (\ref{matSDE}), at each fixed time $x$ we can change to
the diagonal basis for $P_x$ because the noise term is invariant in
distribution and the drift term is equivariant. It\^o's lemma amounts
to formally writing $dp_i = \nabla_{dP} p_i +\frac{1}{2}\nabla
_{dP}^2
p_i$ and using that $dB_{ii}$ are jointly distributed as $\sqrt
{2/\beta} \,db_i$ for $i=1,\ldots,r$ while $\llvert dB_{ij}\rrvert^2 =
dt$ for $j\neq i$. We thus arrive at (\ref{evSDE}).

Recall that the evolution of $P$ through a focal point is still
described by an SDE, after changing coordinates. The same is therefore
true of $\pp$ through an explosion; the form (\ref{evSDE2}) is
obtained from (\ref{evSDE}) by an application of It\^o's lemma.

Just as with the usual Dyson's Brownian motion, the $p_i$ are almost
surely distinct at all positive times: $\pp(x)\in\Cc$ for all $x>0$.
One can show this ``no collision property'' holds for any solution of
(\ref{evSDE}), (\ref{evSDE2}), even with an initial condition $\pp
(0)\in\del C_0$. (Technically, one defines an entrance law from $\del
\Cc$ by a limiting procedure.) Since the coefficients are regular
inside $\Cc$, this suffices to prove uniqueness of the law. See
\citet{AGZ}, Section~4.3.1 for a detailed proof in the driftless
case.
\end{pf}


\begin{pf*}{Proof of Theorem~\ref{t.SDE}}
Explosions of $\pp$ as in
Theorem~\ref{t.Dyson} correspond to focal points of $F$ for each
$\lambda$. By Theorem \ref{t.osc}, the total number of explosions $K$
is equal to the number of eigenvalues strictly below $\lambda$.
(Notice that $\pp$ ends up in $C_K$.) For a \emph{fixed} $\lambda$,
translation invariance of the driving Brownian motions $b_i$ allows one
to shift time $x\mapsto x-\lambda/r$ and use (\ref{SDE}) started at
$x_0 = -\lambda/r$. Putting $a=-\lambda$ we have $\P(-\Lambda_k\le
a) = \P(\Lambda_k\ge\lambda) = \P_{a/r,\ww}(K\le k)$ as required.
\end{pf*}

\subsection{PDE and boundary value problem}

We now prove the PDE characterization, Theorem \ref{t.PDE}. We will
need two properties of the eigenvalue diffusion.

%
\begin{lemma}\label{l.dprops} Let $\pp:[x_0,\infty)\to\ol\Cc$
have law $\P_{x_0,\ww}$ as in (\ref{SDE}) and let $K$ be the number
of explosions. Then the following hold:
\begin{enumerate}[(ii)]
\item[(i)] Given $x_0,k$, $\P_{x_0,\ww}(K\le k)$ is increasing in $\ww$
with respect to the partial order $\ww\le\ww'$ given by $w_i\le
w_i'$, $i=1,\ldots,r$.

\item[(ii)]$\P_{x_0,\ww}$-almost surely, $p_1,\ldots,p_r$ remain bounded
below in $C_K$ (after the last explosion), or equivalently in $C_0$ on
the event $\{K=0\}$.
\end{enumerate}
\end{lemma}

\begin{pf} Part (i) is a consequence Theorem~\ref{t.SDE} and
Remark~\ref{r.mon}, the pathwise monotonicity of the eigenvalues
$\Lambda_k$ as a function of the boundary parameter $W$ with respect
to the usual matrix partial order. It can also be seen from the related
fact that the matrix partial order is preserved pathwise by the matrix
Riccati equation (\ref{matSDE}), which implies that a solution started
from $W$ explodes no later than one started from $W'\ge W$. This fact
holds for the $P$ evolution if it holds for the $Q$ evolution (\ref
{Q}), and for the latter it is Theorem IV.4.1 of \citet{Reid2}.

Part (ii) follows from the stronger assertion that $p_i\sim\sqrt{rx}$
as $x\to\infty$. In the $r=1$ case, this is Proposition 3.7 of
{RRV}. Heuristically, the single particle drift linearizes
at the stable equilibrium $\sqrt{rx}$ to $2\sqrt{rx}(\sqrt
{rx}-p_i)$; even with the repulsion terms one expects fluctuations of
variance only $C/\sqrt{x}$. We omit the proof.
\end{pf}

\begin{pf*}{Proof of Theorem \ref{t.PDE}}
Assume the diffusion
representation of Theorem~\ref{t.SDE} for $F_\beta(x;\ww)=\P
(-\Lambda_0\le x)$ on $\R\times\ol C_0$. We first show $F = F_\beta
$ has the asserted properties and afterward argue uniqueness. Writing
$L$ for the space-time generator of (\ref{SDE}), the PDE (\ref{PDE})
is simply the equation $LF = 0$ after replacing $x$ with $x/r$. In
other words, it is the Kolmogorov backward equation for the hitting
probability (\ref{explosions}) (more precisely, the probability of
never hitting $\{w_1=-\infty\}$), which is $L$-harmonic. This extends
to $w_r=+\infty$ by using the local coordinate there; from (\ref
{evSDE2}) one sees that the coefficients remain regular. Although the
diffusivity vanishes at $w_r=+\infty$, the drift does not, and it
follows that $F$ is continuous up to $w_r=+\infty$. The PDE holds even
at points $\ww\in\del C_0$ with appropriate one-sided derivatives;
notice that the apparent singularity in the ``Dyson term'' of the PDE
is in fact removable for $F$ regular and symmetric in the $w_i$. [For a
toy version, consider a function $f:\R\to\R$ that is twice
differentiable and even; then $f'$ is odd and $f'(w)/w$ is continuous
with value $f''(0)$ at $w=0$. These functions form the domain of the
generator of the Bessel process on the half-line $\{w\ge0\}$ in the
same way that symmetric functions form the domain of the generator of
Dyson's Brownian motion on a Weyl chamber.] Finally, the picture can be
copied to $\ww\in(-\infty,\infty]^r$ by symmetry, permuting the $w_i$.

The boundary condition (\ref{PDEBC1}) follows from the monotonicity
property of Lemma~\ref{l.dprops}(i). For fixed $\ww$, $F(x;\ww)\to
1$ as $x\to\infty$ because it is a distribution function in $x$; by
monotonicity in $\ww$, the convergence is uniform over a set of $\ww$
bounded below. To understand the boundary condition (\ref{PDEBC2})
(using $w_1$ in~$\ol C_0$), change to the coordinate $\tilde w_1=1/w_1$
and close the domain to include the ``bottom boundary'' $\{w_1 =
-\infty\}$. Then (\ref{PDEBC2}) becomes an ordinary Dirichlet
condition. While the diffusivity vanishes on this boundary, the drift
is nonzero into the boundary. The hitting probability is therefore
continuous up to the boundary.

For $F^k$, there is the following more general picture. Consider the
PDE in $\ol C_0\cup\cdots\cup\ol C_k$, defined across the seams by
changing coordinates as in (\ref{evSDE2}). Put the boundary condition
(\ref{PDEBC1}) on all the chambers and (\ref{PDEBC2}) on the bottom
of $\ol C_k$. Then the solution is $F^k$ in $\ol C_0$; the reason is
the same as for $F=F^0$, but now using (\ref{explosionsh}) and the
hitting event ``at most $k$ explosions''. Similarly, the solution is
$F^{k-1}$ in $\ol C_1$ and so on down to $F^0$ in $\ol C_k$. Continuity
holds across the seams and (\ref{PDEBCm}) follows after permuting coordinates.

Toward uniqueness, suppose $\tilde F$ is another bounded solution of
the boundary value problem (\ref{PDE})--(\ref{PDEBC2}) on $\R\times
\ol C_0$. With the notation of Theorem~\ref{t.SDE}, $\tilde F(rx; \pp
_x)$ is a local martingale under $\P_{x_0,\ww}$ by the PDE (\ref
{PDE}). It is therefore a bounded martingale. Let $\zeta\in
(x_0,\infty
]$ be the time of the first explosion; optional stopping gives $\tilde
F(rx_0;\ww) = \E_{x_0,\ww} \tilde F (r(\zeta\wedge x);\pp
_{\zeta\wedge x} )$ for all $x\ge x_0$. Taking $x\to\infty$, we
conclude by bounded convergence, the boundary behaviour (\ref
{PDEBC1}), (\ref{PDEBC2}) of $\tilde F$ and Lemma~\ref{l.dprops}(ii)
that $\tilde F(rx_0,\ww) = \P_{x_0,\ww}(\zeta= \infty)$. By
Theorem~\ref{t.SDE}, this probability is $F_\beta(rx_0,\ww)$. One
argues similarly for the higher eigenvalues.
\end{pf*}

\section{Connection with Painlev\'e II}\label{s.painleve}

In Part~I, we used the PDE characterization to give new
proofs of certain Painlev\'e II formulas for the single-parameter (rank
one deformed) distribution functions $F_\beta(x;w)$ in the cases
$\beta= 2,4$, in particular recovering the Painlev\'e II
representations for the corresponding undeformed Tracy--Widom
distributions by taking $w\to\infty$. The Painlev\'e formulas
appeared originally in \citeauthor{BR1} (\citeyear{BR1,BR2}) in a different context;
in the random matrix theory setting, \citet{B} derived them from
the {BBP} result in the case $\beta= 2$ but they are new
for $\beta= 4$ when $w\neq0$ [see \citet{W}].

\citet{B} also derives a Painlev\'e II formula for the
multi-parameter distribution function $F_2(x;w_1,\ldots,w_r)$. While
we do not have a full independent proof at present, we used the
computer algebra system Maple to verify symbolically that it does
indeed satisfy our PDE (\ref{PDE}) at $\beta=2$ for $r = 2,3,4,5$.
Since this article was first posted, a pencil-and-paper proof for all
$r$ was found [\citet{BBaik}]. We first state Baik's formula and then
briefly describe the symbolic computation.

Let $u(x)$ be the Hastings--McLeod solution of the homogeneous Painlev\'
e II equation
%
%
\begin{equation}
\label{PII} u'' = 2u^3 + xu,
\end{equation}
characterized by
\[
u(x)\sim\Ai(x)\qquad\mbox{as }x\to+\infty,
\]
where $\Ai(x)$ is the Airy function. Put
\vspace{-6pt}
%
%
\begin{eqnarray}
\label{v} v(x) &=& \int_x^\infty u^2,
\\
\label{EF} E(x) &=& \exp\biggl(-\int_x^\infty u
\biggr),\qquad F(x) = \exp\biggl(-\int_x^\infty v
\biggr).
\end{eqnarray}
Next, define two functions $f(x,w)$, $g(x,w)$ on $\R^2$, analytic in
$w$ for each fixed $x$, by the first- order linear ODEs
%
%
\begin{equation}
\label{w_lax} \frac{\del}{\del w} %
\pmatrix{f
\cr
g} %
=
\pmatrix{u^2&-wu-u'
\cr
-wu+u'&w^2-x-u^2}
\pmatrix{f
\cr
g} %
\end{equation}
and the initial conditions
\[
f(x,0) = E(x) = g(x,0).
\]
Equation (\ref{w_lax}) is one member of the Lax pair for the Painlev\'
e II equation. The other member of the pair is
%
%
\begin{equation}
\label{x_lax}\frac{\del}{\del x} %
\pmatrix{f
\cr
g} %
=
\pmatrix{0&u(x)
\cr
u(x)&-w} %
\pmatrix{f
\cr
g},
\end{equation}
which holds for each fixed $w\in\R$. The consistency condition for
the over-determined system (\ref{w_lax}), (\ref{x_lax}) (i.e., that
the partials commute) is the Painlev\'e II equation (\ref{PII}).
The functions $f,g$ can also be defined in terms of an associated
Riemann--Hilbert problem [see, e.g., \citet{B}].

Baik's formula is
%
%
\begin{equation}
\label{Baik} F_2(x;w_1,\ldots,w_r) = F(x)
\frac{\det( (w_i+ \del/(\del x) )^{j-1}f(x,w_i) )_{1\le i,j \le r}}{\prod_{1\le
i<j\le r}(w_j-w_i)}.
\end{equation}

Our symbolic verification for small values of $r$ consisted of the
following steps. The differential relations given by (\ref
{PII})--(\ref{x_lax}) were encoded as formal substitution rules. The
determinant in (\ref{Baik}) was expanded (this step becomes
problematic for larger~$r$\dots!) and the result plugged into our
PDE (\ref{PDE}). The substitution rules were then applied repeatedly.
Finally, the result was factored using Maple's built-in command. Each
time, the output contained the factor
\[
v+u^4-\bigl(u'\bigr)^2+xu^2,
\]
which vanishes identically: differentiate and apply (\ref{PII}) to see
it is constant, and take $x\to\infty$ to see the constant is zero.


\section*{Acknowledgements}
Alex Bloemendal would like to thank Percy Deift for valuable comments and Jinho Baik, Alexei
Borodin, Peter Forrester, Brian Rider, Craig Tracy, Benedek Valko and
Dong Wang for interesting and helpful discussions.


%

\printaddresses
\end{document}